\documentclass[11pt,reqno]{amsart}
\usepackage[colorlinks=true,allcolors=blue,backref=page]{hyperref}
\usepackage{color}
\usepackage{amsmath, amssymb, amsthm}

\usepackage{mathrsfs}
\usepackage{mathtools}
\usepackage[noabbrev,capitalize,nameinlink]{cleveref}
\crefname{equation}{}{}
\usepackage{fullpage}
\usepackage[noadjust]{cite}
\usepackage{graphics}
\usepackage{pifont}
\usepackage{tikz}
\usepackage{bbm}
\usepackage[T1]{fontenc}
\DeclareMathOperator*{\argmax}{arg\,max}

\usetikzlibrary{arrows.meta}

\usepackage{environ}
\usepackage{framed}
\usepackage{url}
\usepackage[linesnumbered,ruled,vlined]{algorithm2e}
\usepackage[noend]{algpseudocode}
\usepackage[labelfont=bf]{caption}
\usepackage{cite}
\usepackage{framed}
\usepackage[framemethod=tikz]{mdframed}
\usepackage{appendix}
\usepackage{graphicx}
\usepackage[textsize=tiny]{todonotes}
\usepackage{tcolorbox}
\usepackage{enumerate}
\usepackage[shortlabels]{enumitem}
\allowdisplaybreaks[1]

\apptocmd{\sloppy}{\hbadness 10000\relax}{}{} 

\crefname{algocf}{Algorithm}{Algorithms}

\crefname{equation}{}{} 
\crefname{conjecture}{Conjecture}{Conjectures} 
\AtBeginEnvironment{appendices}{\crefalias{section}{appendix}} 

\crefformat{enumi}{#2#1#3}
\crefrangeformat{enumi}{#3#1#4 to~#5#2#6}
\crefmultiformat{enumi}{#2#1#3}%
{ and~#2#1#3}{, #2#1#3}{ and~#2#1#3}

\usepackage[color,final]{showkeys} 

\colorlet{refkey}{orange!20}
\colorlet{labelkey}{blue!30}

\crefname{algocf}{Algorithm}{Algorithms}

\numberwithin{equation}{section}
\newtheorem{theorem}{Theorem}[section]
\newtheorem{proposition}[theorem]{Proposition}
\newtheorem{lemma}[theorem]{Lemma}

\crefname{claim}{Claim}{Claims}

\newtheorem*{question*}{Question}

\theoremstyle{definition}
\newtheorem{definition}[theorem]{Definition}

\newtheorem*{definition*}{Definition}

\theoremstyle{remark}
\newtheorem*{remark}{Remark}

\newcommand{\snorm}[1]{\lVert#1\rVert}

\newcommand{\mb}{\mathbb}
\newcommand{\mbf}{\mathbf}
\newcommand{\mbm}{\mathbbm}
\newcommand{\mc}{\mathcal}
\newcommand{\mf}{\mathfrak}

\newcommand{\ol}{\overline}
\newcommand{\on}{\operatorname}
\newcommand{\wh}{\widehat}
\newcommand{\wt}{\widetilde}

\newcommand{\eps}{\varepsilon}
\newcommand{\imod}[1]{~\mathrm{mod}~#1}

\allowdisplaybreaks

\title{Vinogradov's theorem for primes with Restricted digits}

\author[A1]{James Leng}
\address{Department of Mathematics, UCLA, Los Angeles, CA 90095, USA}
\email{jamesleng@math.ucla.edu}

\author[A2]{Mehtaab Sawhney}
\address{Department of Mathematics, Columbia University, New York, NY 10027}
\email{m.sawhney@columbia.edu}

\begin{document}

\begin{abstract}
Let $g$ be sufficiently large, $b\in\{0,\ldots,g-1\}$, and $\mc{S}_b$ be the set of integers with no digit equal to $b$ in their base $g$ expansion. We prove that every sufficiently large odd integer $N$ can be written as $p_1 + p_2 + p_3$ where $p_i$ are prime and $p_i\in \mc{S}_b$.
\end{abstract}

\maketitle

\section{Introduction}\label{sec:introduction}

A celebrated theorem of Vinogradov (see \cite[Chap.~2]{Dav00}) established that every sufficiently large odd positive integer is the sum of three primes. Meanwhile, recent work of Maynard \cite{May17,May22} has established the existence of primes with restricted digits. We establish a common generalization of these results given that the underlying base is sufficiently large. As is standard, $\Lambda(\cdot)$ denotes the von--Mangoldt function. 

\begin{theorem}\label{thm:main}
Fix $A\ge 1$, $g$ be sufficiently large and $b\in \{0,\ldots,g-1\}$. Let $\mbf{P}$ denote the set of primes and $\mc{S}_b$ denote the set of positive integers with no $b$ in their base-$g$ expansion. Then 
\[\sum_{\substack{x_1+x_2+x_3 = N\\x_i\in \mc{S}_b\cap \mbf{P}}}\prod_{i=1}^{3}\Lambda(x_i) = \big(1 + O_A((\log N)^{-A})\big)\prod_{\substack{p\\(p,g) =1}}\bigg(1 - \frac{p\mbm{1}_{p|N} - 1}{(p-1)^3}\bigg)\prod_{p|g}\bigg(\frac{p}{p-1}\bigg)^3\sum_{\substack{x_1+x_2+x_3 = N\\x_i\in \mc{S}_b\\(x_i,g) = 1}}1. \]  
\end{theorem}
\begin{remark}
A conservative estimate for $g$ is $g \ge 2^{100000}$; this can be improved by more carefully going through our proof, but is unlikely to give a ``civilized'' lower bound.
\end{remark}

We remark that there have been a number of generalizations of Vinogradov's three primes theorem to various subsets of primes. To highlight just a small subset of these results, we mention work of Balog and Friedlander \cite{BF92} on representations with Piatetski--Shapiro primes, work of Matom\"aki and Shao \cite{MS17} on representations with Chen primes, work of Grimmelt \cite{Gri22} to Fouvry-Iwaniec primes and a long line of work considering Vinogradov's theorem with nearly equal primes with current record coming from work of Matom\"aki, Maynard, and Shao \cite{MMS17}. We also mention recent work of Green \cite{Gre24} establishing a variant of Waring's problem using only $k$-th powers of integers with restricted digits.

\subsection{Sketch of proof}
Given the statement of the result, it is of no surprise that the proof proceeds via Fourier analytic methods. There are, however, a few subtleties one needs to account for. For instance, note that the ``main term'' contains expressions of the form 
\[\sum_{\substack{x_1+x_2+x_3 = N\\x_i\in \mc{S}_b}} 1.\]
A crucial issue is this expression varies dramatically as $N$ ranges over a dyadic range (by a $N^{\Omega_{g}(1)}$ factor). Formulated alternatively, there are a large number of Fourier coefficients which contribute to the $L^3$-norm of the Fourier transform of $\mbm{1}_{\mc{S}_b}$. This also appears to prevent the use of a Fourier transference approach pioneered by Green \cite{Gre05} in the context of Roth's theorem for the primes and which has been developed in a slew of the aforementioned works including \cite{MS17,MMS17,Gri22}. 

The crucial starting point in our work, as in essentially all work regarding digital properties of primes including that of Mauduit and Rivat \cite{MR10}, Bourgain \cite{Bou15}, and Maynard \cite{May17,May22}, is to exploit that the Fourier transform of $\mbm{1}_{\mc{S}_b}$ exhibits on average better than square--root cancellation. In particular
\[\int_{\mbf{R}/\mbf{Z}}\bigg|\sum_{0\le n<g^{k}}e(n\Theta)\mbm{1}_{n\in \mc{S}_b}\bigg|~d\Theta\ll_\eps g^{\eps k}\]
where $\eps\to 0$ as $g\to \infty$. This is much smaller than a bound of $g^{k/2}$ (say) which one may expect in light of square--root cancellation heuristics. Our second key ingredient is work of Green \cite[Section~2]{Gre22} which uses grand zero--density estimates in order to provide a certain ``Fourier--model'' for the von--Mangoldt function. In particular, Green finds $\Lambda^{\#}(n)$ such that 
\[\sup_{\Theta\in \mbf{R}/\mbf{Z}}\big|\sum_{1\le n\le N}(\Lambda(n) - \Lambda^{\#}(n))e(n\Theta)\big| \le N^{1-\Omega(1)}\]
where $\Lambda^{\#}(n)$ comes via carefully accounting for various corrections coming from zeros of Dirichlet $L$-functions. The key point is that via direct Fourier expansion (see \cref{lem:expan}), this extends to 
\[\sup_{\Theta\in \mbf{R}/\mbf{Z}}\big|\sum_{1\le n\le N}\mbm{1}_{n\in \mc{S}_b}(\Lambda(n) - \Lambda^{\#}(n))e(n\Theta)\big| \le N^{1-\Omega(1)}.\]
Thus via standard maneuvers, it suffices to instead bound 
\[\sum_{\substack{x_1+x_2+x_3 = N\\x_i\in \mc{S}_b}}\prod_{i=1}^{3}\Lambda^{\#}(x_i).\]

Under the assumption of GRH, we have that $\Lambda^{\#}$ is precisely 
\[\sum_{q\le Q}\frac{\mu(q)}{\phi(q)}c_q(n),\]
where $c_q(n) = \sum_{a\in (\mbf{Z}/q\mbf{Z})^{\ast}}e\big(\frac{an}{q}\big)$; here $Q$ is taken to be a small power of $N$. Without GRH, there are additional terms coming from possible zeros of $L$ functions which we discuss at the end of the sketch. The remaining issue therefore is to estimate 
\[\sum_{\substack{x_1+x_2+x_3 = N\\x_i\in \mc{S}_b}}\prod_{i=1}^{3}\Lambda^{\#}(x_i)\]
and analyze which terms are negligible compared to the main term. 

The key idea is to prove that fixing $x_3$, all but certain ``obvious'' major terms vanish; in particular for fixed $x_3$ one can estimate 
\[\sum_{\substack{x_1+x_2 = N-x_3\\x_i\in \mc{S}_b}}\prod_{i=1}^{2}\Lambda^{\#}(x_i)\]
extremely well. The trick in order to do this is to note that the distribution of $x_1$ given that $x_1 + x_2 = T$ and $x_i\in \mc{S}_b$ can be decomposed into a sequence of ``restricted digit measures''. In particular, if one reveals the sum of the $j$-th digits of $x_1$ and $x_2$ for all $j$ then the distribution of $x_1$ is given by choosing the $j$-th digit from a certain set $\mc{B}_j$. Each of these submeasures has an extremely nice Fourier product structure and allows us to prove $L^{1}$ Fourier transform and large sieve inequalities directly via replicating the techniques of Maynard \cite{May17}. 

To illustrate this, we require some notation. Let $x_{i, j}$ be the $j$th digit of $x_i$ in base $g$ and $T_j$ be the $j$th digit of $T$. Note that $x_{1, 1} + x_{2, 1} = T_1 + ig$ where $i \in \{0, 1\}$ represents a ``carry'' in the digit addition.  We then have $x_{1, 2} + x_{2, 2} = T_2 - i \pmod{g}$. We are now effectively in a situation where we have replaced $T$ with $(T - T_1)/g - i$. Notice that if we fix the values of the carries $i_j$ (with $i_0 = 0$) beforehand, the probability distribution decouples as
\[i_{j - 1} + x_{1, j} + x_{2, j} = T_j + gi_{j}\]
Here, it is easy to see that conditioned on $i_j$, the probability distribution of $x_{1, j}$ is according to a product measure on an interval $\mathcal{B}_j$ which is missing at most two digits. Thus, the actual probability distribution of $x_{1, j}$ decomposes as a superposition of many such product distributions. 

A key task in this paper is proving that the decomposition induced by the above procedure is sufficiently non-degenerate (e.g. the associated sets $\mc{B}_j$ are relatively large). This is accomplished via a series of combinatorial lemmas proven in \cref{sec:combo}. Given these non--degeneracy conditions one can then apply the methods of \cite{May17} to derive the necessary arithmetic information on various measures such as $x_1$ conditional on $x_1+x_2 = T$.

Since this procedure of conditioning on the ``carry'' to construct the product measure is rather cumbersome to describe formally, and is performed several times throughout the proof, we shall take the liberty of describing this procedure as simply ``revealing'' (or conjugations of ``reveal'') digits. 

The handling of terms corresponding to various Dirichlet $L$-function correction terms are more complicated as certain character terms arise naturally and we only establish cancellation if $N-x_3$ is sufficiently co-prime to the conductor of zeros under consideration. However, as we only need an ``averaged in $x_3$'' result, one can in fact establish the necessary estimate for almost all $x_3$ and complete the proof. A final remark here is that this reduction to only considering ``binary correlation averages'' is possible due to the ``better than square--root cancellation'' and explicitly used in the work of, say, Maynard \cite{May17,May22}.

\subsection{Notation}
For the remainder of the paper, set $M = g^{k}$ such that $g^{k-1}\le N<g^{k}$ and let $\mc{S}_{b,M} = \mc{S}_{b}\cap\{0,1,\ldots,M-1\}$. Throughout, we use that $[a,b] = \{a,a+1,\ldots,b\}$ if $a,b\in \mbf{Z}$ and $[n] = \{0, 1, \dots, \lfloor n\rfloor - 1\}$ for $n \in \mbf{R}$; we will not require the corresponding expression for a closed interval. Furthermore throughout we use the notation $A\sim B$ to denote that $B/2\le A<2B$ (e.g., $A$ lies within a factor of $2$ of $B$) and $X\sim \mu$ to denote $X$ is drawn from the measure corresponding to $\mu$; the usage will always be clear from context. For $x \in \mbf{N}$, the quantity $\tau(x)$ denotes the number of the divisors of $x$ and let $v_p(x)$ denote the largest integer $e$ such that $p^{e}|x$. Finally let $\on{Ber}(p)$ denote a random variable which is $1$ with probability $p$ and $0$ otherwise and $\on{Bin}(n,p)$ be the sum of $n$ independent copies of $\on{Ber}(p)$.

As is standard, we let $e(\Theta) = e^{2\pi i\Theta}$. Throughout we operate with the Fourier transform from $\mbf{Z}\to \mbf{C}$. In particular, given a finitely supported function $f:\mbf{Z}\to \mbf{C}$ we define 
\[\wh{f}(\Theta) :=\sum_{x\in \mbf{Z}}f(x) e(-x\Theta). \]
In this normalization, we have that the Fourier inversion formula is 
\[f(x) = \int_{\mbf{R}/\mbf{Z}}\wh{f}(\Theta)e(x\Theta)~d\Theta\]
and Parseval/Plancherel takes the form
\[\sum_{x\in \mbf{Z}}\ol{f(x)}g(x) = \int_{\mbf{R}/\mbf{Z}}\ol{\wh{f}(\Theta)}\wh{g}(\Theta)~d\Theta.\]

Finally we use standard asymptotic notation for functions on $\mb{N}$ throughout. Given functions $f=f(x)$ and $g=g(x)$, we write $f=O(g)$, $g = \Omega(f)$, $f\ll g$, or $g\gg f$ to mean that there is a constant $C$ such that $|f(x)|\le Cg(x)$ for sufficiently large $x$. In particular, $\Omega$ will not be used to denote the prime divisor counting function. Subscripts in this notation indicate the implied constant depends on the subscripts. 

\subsection{Acknowledgements}
The second author thanks Ashwin Sah for initial discussions regarding this problem. We thank Zachary Chase, Ben Green, and Zach Hunter for useful comments and corrections. The authors would also like to thank the anonymous referees for helpful comments.

Leng was supported by NSF Graduate Research Fellowship Grant No.~DGE-2034835. This research was conducted during the period Sawhney served as a Clay Research Fellow.

\section{Combinatorics of digits representations}\label{sec:combo}

The crucial input in proving that various decompositions give rise to ``dense cubes'' is proving that the number of representations of $N$ is not overly sensitive to small changes in the final digit. 

\begin{lemma}\label{lem:init}
There exists a constant $C = C_{\ref{lem:init}}\ge 1$ such that the following holds. Let $g\ge C$, $g^{r}\le T<g^{r+1}$ for an integer $r\ge 2$. Then
\[\max_{j_1,j_2\in [0,3]}\frac{\#\{\sum_{i=1}^{3}x_i = T-j_1:x_i\in \mc{S}_b\}}{\#\{\sum_{i=1}^{3}x_i = T-j_2:x_i\in \mc{S}_b\}}\le C.\]
\end{lemma}
\begin{remark}
The analog of this lemma is unfortunately not true with a ``bivariate'' summation $x_1 + x_2$ and $b = 1$. In particular if $x_1 + x_2 = 2g^r - 1$, then one of $x_1$ or $x_2$ must have leading digit $1$.
\end{remark}
\begin{proof}
We first handle the initial case where $r = 2$ via casework. Note that 
\[\#\bigg\{\sum_{i=1}^{3}x_i = T:x_i\in \mc{S}_b\bigg\}\le \frac{T^2}{2}.\] We prove that for $g^{2}\le T<g^{3}$, we may achieve a positive fraction of the upper bound. Let $\mc{S}_b^{\ast}$ denote the set of integers with no $0$ or $b$ in their base $g$ expansions; note that
\[\#\big\{x\le g^2-1:x\notin \mc{S}_b^{\ast}\big\}\le 4g.\]
Thus for $g^2/4\le T\le 11g^2/4$, we have
\[\#\bigg\{\sum_{i=1}^{3}x_i = T:x_i\in \mc{S}_b, x_i\le g^2\bigg\}\gg T^2.\]

For the remainder of the initial case, let $T = a_1g^2 + a_2$ with $a_1\in [2,g-1]$ and $a_2\in [0,g^2-1]$. As $g^2\le a_2 + g^2 \le 2g^2 - 1$, if 
\[\big\{x_1 + x_2 + x_3 = (a_1 - 1): x_i\in \{0\}\cup \mc{S}_b\big\}\gg a_1^2,\]
the initial case follows. Using that $a_1 - 1 = (a_1 - 1) + 0 + 0 = (a_1 - 3) + 1 + 1$ and that there is at most $1$ missing digit, the inequality is trivially verified for $a_1 - 1\ge 2$. Thus the only remaining case is $a_1 = 2$ and that $a_2\ge 3g^2/4$. If $b \neq 1$, we have sufficient many representations via $a_1 - 1 = 1 + 0 + 0$. Else if $b = 1$, note that $a_1 = 2 + 0 + 0$ and we have sufficiently many representations with $x_1,x_2\le g^2-1$ and $2g^2\le x_3<3g^2$. 

We now proceed via induction. We will have $x_i\in \mc{S}_b$ throughout the remainder of the proof even when not specified. Let $T = T'g + s$ where $s\in[0,g-1]$ and let $j\in [0,3]$. Note that 
\[\#\big\{x_1+x_2 + x_3 = T-j\big\} = \sum_{k=0}^{2}\#\big\{x_1+x_2+x_3 = T'-k\big\}\cdot\#\big\{x_1+x_2+x_3 = kg+s-j\big\}.\]
For $s\in[0,g-1]$, we have
\begin{align*}
s(s+1)/2-3s&\le \#\{x_1+x_2+x_3 = s\}\le s(s+1)/2\\
g^2&\ll \#\{x_1+x_2 +x_3= g + s\}\ll g^2\\
(g-s-1)(g-s)/2 - 12(g-s)&\ll \#\{x_1+x_2 +x_3= 2g + s\}\ll (g-s-1)(g-s)/2.
\end{align*}

Note that if $100\le \min(s,g-s)$, the desired result follows immediately. We now handle the case where $s\le 100$; the remaining case is analogous. If 
\[\argmax_{j\in [0,2]}\#\big\{x_1+x_2+x_3 = T' - j\big\} \in [1,2],\]
the result follows immediately. Else let 
\[M \ge \max_{j,j'\in [0,3]}\frac{\#\{x_1+x_2+x_3 = T'-j\}}{\#\{x_1+x_2+x_3 = T'-j'\}}\]
and we have 
\[\max_{j_1,j_2\in \{0,1,2,3\}}\frac{\#\{\sum_{i=1}^{3}x_i = T-j_1:x_i\in \mc{S}_b\}}{\#\{\sum_{i=1}^{3}x_i = T-j_2:x_i\in \mc{S}_b\}}\ll \frac{M + g^2}{g^2}\le M/2\]
provided $M$ is larger than an absolute constant. This completes the induction and hence the proof of the lemma.
\end{proof}

We next require that the set of solutions to $x_1 + x_2 + x_3 = N$ is not too sparse via noting that $\mc{S}_b + \mc{S}_b + \mc{S}_b$ is well distributed $\imod g$. 
\begin{lemma}\label{lem:sol-bound}
Let $g\ge C_{\ref{lem:sol-bound}}$ and $T\ge g^2$. Then
\[\#\{x_1 + x_2 + x_3 = T:x_i\in \mc{S}_b\}\ge g^{-C_{\ref{lem:sol-bound}}}\cdot T^{2\log(1-3/g)}.\]
\end{lemma}

\begin{proof}
Throughout the proof, we drop the condition that $x_i\in \mc{S}_b$. We proceed by induction on $T$; for $g^{2}\le T\le 2g^{3}$, the proof of \cref{lem:init} implies that 
\[\#\{x_1 + x_2 + x_3 = T\}\ge 1.\]
For larger $T$, let $T = gT' + s$ with $s\in [0,g-1]$ and we have 
\begin{align*}
\{x_1 + x_2 + x_3 = T\} &= \sum_{j=0}^{2}\{x_1+x_2+x_3 = s + jg\}\{x_1 + x_2 + x_3 = T' -j\}\\
&\ge \{x_1 + x_2 + x_3\equiv s\imod g\} \cdot \min_{T^{\ast}\in [T'-2,T']}\{x_1 + x_2 + x_3 = T^{\ast}\}\\
&\ge (g^2 - 3g) \cdot \min_{T^{\ast}\in [T'-2,T']} \{x_1 + x_2 + x_3 = T^{\ast}\}.
\end{align*}
A routine induction argument completes the proof. 
\end{proof}

The next lemma while a bit technical to state captures a certain ``stochastic domination'' criterion which will be used throughout the paper. 

\begin{lemma}\label{lem:chernoff}
There exists a constant $C = C_{\ref{lem:chernoff}}\ge 1$ such that the following holds. 

Suppose that $g^{r}\le T<g^{r+1}$ with $g\ge C$, $t\ge 0$ and take 
\[x_i = \sum_{j=0}^{k}x_{i,j}g^{j}\]
for $x_{i, j} \in [0, g - 1]$. Let $\mc{S}_j\subseteq [0,g-1]^2$ for $j\in [k/2]$ and 
\[Y_j = \mbm{1}[(x_{1,j},x_{2,j})\in \mc{S}_j].\]
Then for any $S\subseteq [k/2]$, we have
\[\mb{P}_{\substack{x_1 + x_2 + x_3 = T\\x_i\in \mc{S}_b}}\bigg[\sum_{j\in S}Y_j\ge t\bigg]\le \mb{P}\bigg[\sum_{j\in S}\on{Ber}\bigg(\min\bigg(\frac{C\big|\mc{S}_j\big|}{g^2},1\bigg)\bigg)\ge t\bigg].\]
Furthermore, let $\mc{S}_j'\subseteq [0,g-1]$ for $j\in [k/2]$ and 
\[Z_j = \mbm{1}[x_{1,j}\in \mc{S}_j'].\]
Then for any $S\subseteq [k/2]$, we have
\[\mb{P}_{\substack{x_1 + x_2 + x_3 = T\\x_i\in \mc{S}_b}}\bigg[\sum_{j\in S}Z_j\ge t\bigg]\le \mb{P}\bigg[\sum_{j\in S}\on{Ber}\bigg(\min\bigg(\frac{C\big|\mc{S}_j'\big|}{g},1\bigg)\bigg)\ge t\bigg].\]
\end{lemma}
\begin{remark}
Note that this lemma converts various results regarding the digits patterns of $x_i$ into bounds for binomial random variables. 
\end{remark}
\begin{proof}
We first note that the second part of the lemma follows from the first via $\mc{S}_j = \mc{S}_j'\times [0,g-1]$. 

To prove the first part of the lemma, it suffices to prove that
\[\mb{P}_{\substack{x_1 + x_2 + x_3 = T\\x_i\in \mc{S}_b}}\bigg[Y_j = 1|(x_{1,\ell},x_{2,\ell},x_{3,\ell})\text{ for }0\le \ell\le j-1\bigg]\ll \frac{|\mc{S}_j|}{g^2}.\]
For each $s\in [0,g-1]$, there are between $g^2-3g$ and $g^2$ tuples with $x_i'\in [0, g - 1]$ such that $x_1' + x_2' + x_3' = s \imod g$. Furthermore, note that given $T$ and $(x_{1,\ell},x_{2,\ell},x_{3,\ell})$ for $0\le \ell\le j-1$ that $x_{1,j} + x_{2,j} + x_{3,j} \imod g$ is deterministic. The desired result then follows from \cref{lem:init}.
\end{proof}

We now state the Chernoff--Hoeffding bound which will be used throughout the paper.
\begin{lemma}\label{lem:chernoff-hoeff}
For $p\le p'\le 1$, we have
\begin{align*}
\mb{P}_{X\sim \on{Bin}(n,p)}[X\ge np']&\le \exp\bigg(-n\bigg(p'\log\bigg(\frac{p'}{p}\bigg) + (1-p')\log\bigg(\frac{1-p'}{1-p}\bigg)\bigg)\bigg).
\end{align*}
\end{lemma}

\section{Fourier analytic estimates for restricted digit sets}\label{sec:rest-digit}
In this section, we consider certain generalizations of Maynard \cite[Section~5]{May22} to sets of integers where the digit set for each position may differ. The proof closely follows those in \cite[Section~5]{May22} and that only very mild conditions on the digit set are required is essentially present in the discussion in \cite[Section~9]{May22}.
\begin{definition}\label{def:product-set}
Let $\vec{\mc{B}} = (\mc{B}_0,\mc{B}_1,\ldots,\mc{B}_{k-1})$ be subsets of $\{0,\ldots,g-1\}$. Define the measure $\mu_{\vec{\mc{B}}}$ to be the sum of Dirac masses (each of weight $1$) on the set of integers 
\[z = \sum_{j=0}^{k-1}z_jg^{j}\]
where $z_j\in \mc{B}_j$.
\end{definition}
Throughout this section we will be concerned with the cases where either:
\begin{itemize}
    \item $\mc{B}_i$ is a singleton
    \item $\mc{B}_i$ is an interval $\{a_1,a_1+1,\ldots,a_2\}$ except at most $2$ elements have been removed. 
\end{itemize}
Note that the first case is technically subsumed by the second; the distinction between intervals and ``singletons'' will play a larger role in the subsequent analysis. 

The first estimate we require is an $L^{1}$-estimate of the Fourier transform; this is analogous to \cite[Lemma~5.1]{May22}.
\begin{lemma}\label{lem:L1}
There exists a constant $C = C_{\ref{lem:L1}}\ge 1$ such that the following holds. Suppose that $z\sim \mu_{\vec{\mc{B}}}$ where $\mc{B}_i$ are intervals missing at most $2$ elements. We have that 
\[\int_{\Theta\in \mbf{R}/\mbf{Z}}\big|\wh{\mu_{\vec{\mc{B}}}}(\Theta)\big|~d\Theta \le (C\log g)^{k}.\]
\end{lemma}
\begin{proof}
Note that
\[\big|\wh{\mu_{\vec{\mc{B}}}}(\Theta)\big| = \prod_{j=0}^{k-1}\big|\sum_{z_j\in\mc{B}_j}e\big(z_jg^{j}\Theta\big)\big|.\]
Via summing a geometric series, we have that
\[\sum_{z_j\in\mc{B}_j}|e\big(z_jg^{j}\Theta\big)|\le \min\big(|\mc{B}_j|, (2 + 2\snorm{g^{j}\Theta}_{\mbf{R}/\mbf{Z}}^{-1})\big).\]
If $\Theta = \sum_{i=1}^{\infty}\Theta_ig^{-i}$ where $|\Theta_i|\in [0,g-1]$, then $\snorm{g^{i}\Theta}_{\mbf{R}/\mbf{Z}} = \snorm{\Theta_i/g + \eps_i}_{\mbf{R}/\mbf{Z}}$ with $0\le \eps_i< 1/g$. The desired bound then follows (via summing a harmonic series) as 
\begin{align*}
\int_{\Theta\in \mbf{R}/\mbf{Z}}\big|\wh{\mu_{\vec{\mc{B}}}}(\Theta)\big|~d\Theta &\le g^{-k}\sum_{\Theta_1,\ldots,\Theta_k\in [0,g-1]}\prod_{i=1}^{k}\min\bigg(g, 2 + 2\max\bigg(\frac{g}{\Theta_i + 1},\frac{g}{g-1-\Theta_i}\bigg)\bigg)\\
&\le (O(\log g))^k. \qedhere
\end{align*}
\end{proof}

We next prove the analog of \cite[Lemma~10.5]{May17}; the proof is essentially identical modulo notational differences. 
\begin{lemma}\label{lem:large-sieve}
There exists a constant $C = C_{\ref{lem:L1}}\ge 1$ such that the following holds. Let $Q\ge d\ge 1$ be integers and suppose that $z\sim \mu_{\vec{\mc{B}}}$ where  $\mc{B}_i$ are intervals missing at most $2$ elements. Then
\[\sup_{\beta\in \mbf{R}}\sum_{\substack{b\sim Q\\ d|b}}\sum_{\substack{0<a<b\\(a,b) = 1}}\bigg|\wh{\mu_{\vec{\mc{B}}}}\bigg(\frac{a}{b} + \beta\bigg)\bigg|\ll_{g}\inf_{0\le t\le k-1}\prod_{t\le j\le k-1}|\mc{B}_j| \cdot \bigg(\frac{Q^2}{d} \cdot (C\log g)^t + (Cg\log g)^t\bigg).\]
\end{lemma}
\begin{proof}
For $t \in [k]$ fixed, let $\vec{\mc{B}}' = (\mc{B}_0,\ldots,\mc{B}_{t-1})$. We have that 
\[\big|\wh{\mu_{\vec{\mc{B}}}}(\Theta)\big|\le \big|\wh{\mu_{\vec{\mc{B}}'}}(\Theta)\big| \cdot \prod_{t\le j\le k-1}|\mc{B}_j|.\]
Thus, by replacing $\mc{B}$ with $\mc{B}'$ it suffices to consider the case when $t = k-1$.

Note that the set of fractions $a/b$ with $d|b$, $(a,b) = 1$, and $b\sim Q$ are $\Omega(d/Q^2)$ well-spaced. Furthermore, by the fundamental theorem of calculus, for all $\Theta\in \mbf{R}/\mbf{Z}$ and $\delta>0$, we have that
\[\big|\wh{\mu_{\vec{\mc{B}}}}(\Theta)\big|\lesssim \frac{1}{\delta}\int_{-\delta}^{\delta}\big|\wh{\mu_{\vec{\mc{B}}}}(\Theta+s)\big|~ds + \frac{1}{\delta}\int_{-\delta}^{\delta}\big|\wh{\mu_{\vec{\mc{B}}}}'(\Theta+s)\big|~ds;\]
here $\wh{\mu_{\vec{\mc{B}}}}'(\Theta) = \frac{d}{dt}\wh{\mu_{\vec{\mc{B}}}}(t)|_{t = \Theta}$. Taking $\delta = d/Q^2$ and summing over all fractions, we have that 
\begin{align*}
\sup_{\beta\in \mbf{R}}&\sum_{\substack{b\sim Q\\ d|b}}\sum_{\substack{0<a<b\\(a,b) = 1}}\bigg|\wh{\mu_{\vec{\mc{B}}}}\bigg(\frac{a}{b} + \beta\bigg)\bigg|\ll \frac{Q^2}{d}\cdot \int_{\mbf{R}/\mbf{Z}}\big|\wh{\mu_{\vec{\mc{B}}}}(s)\big|~ds + \int_{\mbf{R}/\mbf{Z}}\big|\wh{\mu_{\vec{\mc{B}}}}'(s)\big|~ds.
\end{align*}
We are using here that the set of intervals $(a/b-\delta,a/b+\delta)$ with $a$ and $b$ as specified cover each point $O(1)$ times. The first term is immediately bounded by \cref{lem:L1}. The key point (as in \cite[pg.~12]{May22}) is to observe that the second terms have a near product structure; in particular
\begin{align*}
\big|\wh{\mu_{\vec{\mc{B}}}}'(\Theta)\big| &\le \sum_{\ell = 0}^{k-1}\big|\sum_{z_\ell\in\mc{B}_\ell}z_\ell g^{\ell-1}e\big(z_\ell g^{\ell-1}\Theta\big)\big|\cdot \prod_{\substack{0\le j\le k-1\\j\neq \ell}}\big|\sum_{z_j\in\mc{B}_j}e\big(z_jg^{j-1}\Theta\big)\big|\\
&\le \sum_{\ell = 0}^{k-1}g^{\ell}(g-1)\cdot \prod_{\substack{0\le j\le k-1\\j\neq \ell}}\big|\sum_{z_j\in\mc{B}_j}e\big(z_jg^{j-1}\Theta\big)\big|\le g^{k} \sup_{0\le \ell\le k-1} \prod_{\substack{0\le j\le k-1\\j\neq \ell}}\big|\sum_{z_j\in\mc{B}_j}e\big(z_jg^{j-1}\Theta\big)\big|.
\end{align*}
Observe that the integral of the right hand side of the final inequality is of the form of the Fourier transform of a measure amenable to the hypothesis of \cref{lem:L1}. (One applies \cref{lem:L1} to measures $\mu_{\vec{B}}$ except $\mc{B}_{\ell}$ is changed to any singleton set.) Applying \cref{lem:L1}, we obtain the stated bound.
\end{proof}

\subsection{Divisor functions estimates for sets with missing digits}
We next require bounds on the divisor function when restricting attention to solutions of $x_1+x_2+x_3 = N$ where $x_i\in \mc{S}_b$. This is accomplished via converting \cref{lem:large-sieve} into Type I information and applying Landreau’s inequality. 

\begin{lemma}\label{lem:divisor-function}
Let $g\ge C_{\ref{lem:divisor-function}}$ and $A\ge 1$. We have that 
\[\sum_{\substack{x_1 + x_2 + x_3 = T\\x_i\in \mc{S}_b}}\tau(x_1)^{A}\ll_{g,A} (\log T)^{O_A(1)} \cdot \sum_{\substack{x_1 + x_2 + x_3 = T\\x_i\in \mc{S}_b}}1\]
and 
\[\sum_{\substack{x_1 + x_2 + x_3 = T\\x_i\in \mc{S}_b}}\tau(x_1 + x_2)^{A}\ll_{g,A} (\log T)^{O_A(1)} \cdot \sum_{\substack{x_1 + x_2 + x_3 = T\\x_i\in \mc{S}_b}}1.\]
\end{lemma}
\begin{proof}
We prove only the former; the latter is completely analogous using that $x_1 + x_2 = T-x_3$, replacing instances of $x_1$ with $T-x_3$ in a mechanical fashion and noting that various Fourier estimates on $x_3$ are equivalent to those on $T-x_3$. 
Let us assume that $g^{k}\le T<g^{k+1}$. Write $x_1 = x_1' \cdot x_1^{\ast}$ where $(x_1^{\ast},g) = 1$ and if $p|x_1'$ then $p|g$. As $\tau(x_1) = \tau(x_1') \cdot \tau(x_1^{\ast})$, it suffices to bound 
\[\sum_{\substack{x_1 + x_2 + x_3 = T\\x_i\in \mc{S}_b}}\tau(x_1')^{2A} + \tau(x_1^{\ast})^{2A}.\]
To bound the first term, note that $\phi(g)/g\ge 1/g$ and that $\tau(x_1')^{2A}\le (\ell + 1)^{O_{A,g}(1)}$ where $\ell$ is the number of digits at the end of $x_1$ which are all not coprime to $g$. For $t\le k/2$, the probability that the last $t$ digits are not coprime to $g$ is bounded by $(1 - \Omega(1/g))^{t}$. Thus, 
\begin{align*}
\sum_{\substack{x_1 + x_2 + x_3 = T\\x_i\in \mc{S}_b}}\tau(x_1')^{2A}&\ll_{g} \bigg(T^{-\Omega(1)} + \sum_{1\le t\le k/2}t^{O_{A,g}(1)}(1 - \Omega(1/g))^{t}\bigg)\cdot \bigg(\sum_{\substack{x_1 + x_2 + x_3 = T\\x_i\in \mc{S}_b}}1\bigg)\\
&\ll_{g,A}\sum_{\substack{x_1 + x_2 + x_3 = T\\x_i\in \mc{S}_b}}1.
\end{align*}

For $x_1^{\ast}$, we use Landreau’s inequality (see e.g. \cite[Lemma~3.1(ii)]{TT22}) which states that for any $\eps>0$, we have 
\[\tau(n)\ll_{\eps}\sum_{\substack{d|n\\d\le n^{\eps}}}\tau(d)^{O_\eps(1)}.\]
This implies that if $(n,g) = 1$, then 
\[\tau(n)^{A}\ll_{\eps}\bigg(\sum_{\substack{d|n\\d\le n^{\eps},~(d,g) = 1}}\tau(d)^{O_\eps(1)}\bigg)^{A}\ll_{\eps,A}\sum_{\substack{d_1,\ldots,d_{A}|n\\d_i\le n^{\eps},~(d_i,g) = 1}}\prod \tau(d_i)^{O_{\eps}(1)}\ll_{\eps,A}\sum_{\substack{d|n\\ d\le n^{A\eps},~(d,g) = 1}} \tau(d)^{O_{\eps,A}(1)}.\]
Thus, it suffices to prove that there exists $\eps>0$ such that for any $B\ge 1$ that 
\[\sum_{\substack{x_1 + x_2 + x_3 = T\\x_i\in \mc{S}_b}}\sum_{\substack{d|x_1\\ d\le T^{\eps},(d,g) = 1}}\tau(d)^{B}\ll_{B, g} (\log T)^{O_{B}(1)}\cdot \sum_{\substack{x_1 + x_2 + x_3 = T\\x_i\in \mc{S}_b}}1.\]

Note that 
\begin{align*}
\sum_{\substack{x_1 + x_2 + x_3 = T\\x_i\in \mc{S}_b}}\sum_{\substack{d|x_1\\ d\le T^{\eps},(d,g) = 1}}\tau(d)^{B}&\le \sum_{\substack{x_1 + x_2 + x_3 = T\\x_i\in \mc{S}_b}}\sum_{d\le T^{\eps},(d,g) = 1}\frac{\tau(d)^{B}}{d}\sum_{0\le a<d}e\bigg(\frac{ax_1}{d}\bigg)\\
&\le \sum_{d\le T^{\eps},(d,g) = 1}\frac{\tau(d)^{B}}{d}\sum_{0\le a<d}\bigg|\sum_{\substack{x_1 + x_2 + x_3 = T\\x_i\in \mc{S}_b}}e\bigg(\frac{ax_1}{d}\bigg)\bigg|\\
&\le \sum_{\substack{0\le a<d, (a,d) = 1\\d\le T^{\eps},~(d,g) = 1}}\bigg(\sum_{d'd\le T^{\eps}}\frac{\tau(d')^{B}}{d'}\bigg) \cdot \frac{\tau(d)^{B}}{d}\bigg|\sum_{\substack{x_1 + x_2 + x_3 = T\\x_i\in \mc{S}_b}}e\bigg(\frac{ax_1}{d}\bigg)\bigg|\\
&\le (\log T)^{O_B(1)} \cdot \sum_{\substack{0\le a<d, (a,d) = 1\\d\le T^{\eps},~(d,g) = 1}} \frac{\tau(d)^{B}}{d}\bigg|\sum_{\substack{x_1 + x_2 + x_3 = T\\x_i\in \mc{S}_b}}e\bigg(\frac{ax_1}{d}\bigg)\bigg|.
\end{align*}

It therefore suffices to prove for a scale $1\le L\le T^{\eps}$ that
\[\frac{1}{L}\cdot \sum_{\substack{0\le a<d,~(a,d) = 1\\L/2\le d\le L,~(d,g) = 1}} \bigg|\sum_{\substack{x_1 + x_2 + x_3 = T\\x_i\in \mc{S}_b}}e\bigg(\frac{ax_1}{d}\bigg)\bigg|\ll_{g} \sum_{\substack{x_1 + x_2 + x_3 = T\\x_i\in \mc{S}_b}}1\cdot L^{-1/2}.\]
Using the divisor bound that $\tau(d) = d^{o(1)}$ and summing over dyadic scales completes the proof. 
Taking $S = [k/2]$, via \cref{lem:chernoff} and \cref{lem:chernoff-hoeff} we have 
\begin{equation}\label{eq:bad-event}
\mb{P}_{\substack{x_1 + x_2+x_3 = T\\ x_i\in \mc{S}_b}}\bigg[\sum_{0\le j<k/2}\mbm{1}\bigg[g^{7/8}\ge \min\bigg(\sum_{i=1}^{2}x_{i,j}, 2g - \sum_{i=1}^{2}x_{i,j}\bigg)\bigg]\ge k/16\bigg]\le \exp\big(-\Omega(k/\log g)\big) = T^{-\Omega(1)}.
\end{equation}

Reveal $\big(\sum_{i=1}^{2}x_{i,j}\big)_{0\le j\le k}$ and $x_{1,j}$ for $j>k/2$. We have that by discarding a set with probability $T^{-\Omega(1)}$, we may decompose $(x_1|x_1 + x_2 + x_3 = T)$ into product measures $\mu_{\vec{\mc{B}}}$ where $x_{1,i}$ is uniform on a set $\mc{B}_i$ with $\mc{B}_i$ are intervals minus at most one point, $|\mc{B}_i|\ge g^{4/5}$ for at least $7k/16$ of the sets with $i<k/2$, and $|\mc{B}_i| = 1$ for $i\ge k/2$. This is exactly the form of \cref{def:product-set} with the additional guarantee that $|\mc{B}_i|\ge g^{4/5}$ for at least $7k/16$ of the sets with $i<k/2$. This is accomplished so by conditioning on the carries of the digit addition between $x_1$ and $x_2$; \cref{eq:bad-event} guarantees precisely the required regularity on $|\mc{B}_j|$.

Let $t = \lceil \log L/\log g\rceil$; there exists a segment $\mc{I}\subseteq [k/2]$ with $|\mc{I}| = t$ and for at least $7t/8$ elements in $i\in \mc{I}$ we $|\mc{B}_i|\ge g^{4/5}$. Via partitioning further, we may assume that the remaining $\mc{B}_i$ are singletons and it suffices to prove under these conditions that
\[\frac{1}{L}\cdot \sum_{\substack{0\le a<d,~(a,d) = 1\\L/2\le d\le L,~(d,g) = 1}} \bigg|\sum_{x\sim \mu_{\vec{\mc{B}}}}e\bigg(\frac{ax}{d}\bigg)\bigg|\ll_{g} \prod_{i=0}^{k-1}|\mc{B}_i|\cdot L^{-1/2}.\]
The cases where the underlying $|\mc{B}_i|$ are not sufficiently large (e.g. live in the failure event given in \cref{eq:bad-event}) are handled by taking trivial bounds (and that $\eps$ is sufficiently small). 

Let $\mc{I} = [t_1,t_2]$ (where $t_2 - t_1 = t-1$), $\mc{B}' = (\mc{B}_0,\ldots,\mc{B}_{t_2})$ and note that it suffices to prove 
\[\frac{1}{L}\cdot \sum_{\substack{0\le a<d,~(a,d) = 1\\L/2\le d\le L,~(d,g) = 1}} \bigg|\sum_{x\sim \mu_{\vec{\mc{B}}'}}e\bigg(\frac{ax}{d}\bigg)\bigg|\ll_{g} \prod_{i=0}^{t_2}|\mc{B}_i|\cdot L^{-1/2}.\]
Note that as we have forced $\mc{B}_0,\ldots,\mc{B}_{t_1-1}$ to be singletons, we have that $x = \bar{x} + g^{t_1}\cdot x^{\ast}$ where $x^{\ast}$ is drawn from $\mu_{\vec{\mc{B}}^{\ast}}$ with $\mc{B}^{\ast} = (\mc{B}_{t_1},\ldots,\mc{B}_{t_2})$. Using that $(d,g) = 1$, it suffices to prove that 
\[\frac{1}{L}\cdot \sum_{\substack{0\le a<d,~(a,d) = 1\\L/2\le d\le L,~(d,g) = 1}} \bigg|\sum_{x\sim \mu_{\vec{\mc{B}}^{\ast}}}e\bigg(\frac{ax}{d}\bigg)\bigg|\ll_{g} \prod_{i=t_1}^{t_2}|\mc{B}_i|\cdot L^{-1/2}.\]
This result follows from \cref{lem:large-sieve} (taking $d = 1$ and $Q = L$) as 
\begin{align*}
\frac{1}{L}\cdot \sum_{\substack{0\le a<d,~(a,d) = 1\\L/2\le d\le L,~(d,g) = 1}} \bigg|\sum_{x\sim \mu_{\vec{\mc{B}}^{\ast}}}e\bigg(\frac{ax}{d}\bigg)\bigg|&\ll_g L \cdot (C\log g)^{t} + L^{-1}(Cg\log g)^t\\
&\le \prod_{i=t_1}^{t_2}|\mc{B}_i| \cdot \big((C\log g) g^{-4/7}\big)^{t} + L^{-1}(Cg^{3/7}\log g)^{t}\big)\\
&\ll_{g} \prod_{i=t_1}^{t_2}|\mc{B}_i| \cdot L^{-1/2};
\end{align*}
the final inequality holds provided $g$ is larger than an absolute constant.
\end{proof}

\section{Fourier approximant of the primes and zero-density estimates}

We next define a Fourier approximant of the primes with a \emph{power-saving} error term. The necessary approximant appears in work of Green \cite[Part~II]{Gre22}; we recall a number of definitions from \cite{Gre22}.
\begin{definition}\label{def:zeros}
Let $Q$ and $\sigma_0$ be a positive real parameters. Let $\Xi_{Q}(\sigma_0)\subseteq \mbf{C}$ be the multiset of all $\rho := \beta + i\gamma = 1-\sigma + i\gamma$ with $\sigma\le \sigma_0$ and $|\gamma|\le Q$ which are zeros of a Dirichlet $L$-function $L(s,\chi)$ with conductor at most $Q$.
\end{definition}

We next require a series of definitions in order to succinctly define the required approximant.
\begin{definition}\label{def:aux-func}
Suppose that $\psi$ is a Dirichlet character to modulus $r$. The Gauss sum of $\psi$ is
\[\tau(\psi) := \sum_{b\in (\mbf{Z}/r\mbf{Z})^{\ast}}\psi(b)e\bigg(\frac{b}{r}\bigg).\]

If $\chi$ is a primitive Dirichlet character to modulus $q$, $r$ a positive integer with $q|r$, and $b\in (\mbf{Z}/r\mbf{Z})^{\ast}$, we define
\[c_{\chi}(b,r) := \frac{\chi(b)\mu\big(\frac{r}{q})\ol{\chi(\frac{r}{q})}\ol{\tau(\chi)}}{\phi(r)}.\]
If $q\nmid r$, we set $c_{\chi}(b,r) = 0$. Given this, we set
\begin{equation}\label{eq:ApproximantError}
F_{\chi,Q}(n) :=\sum_{\substack{q|r\\r/q\le Q}}\sum_{b\in (\mbf{Z}/r\mbf{Z})^{\ast}}c_{\chi}(b,r)e\bigg(\frac{bn}{r}\bigg)    
\end{equation}
Furthermore, set
\begin{equation}\label{eq:ApproximantMain}
\Lambda_Q(n) := \sum_{q\le Q}\frac{\mu(q)}{\phi(q)}c_q(n),    
\end{equation}
where $c_q(n) = \sum_{a\in (\mbf{Z}/q\mbf{Z})^{\ast}}e\big(\frac{an}{q}\big)$. We finally define
\begin{equation}\label{eq:Approximant}
\Lambda_{Q,\sigma_0}(n):=\Lambda_{Q}(n) - \sum_{\rho\in \Xi_{Q}(\sigma_0)}n^{\rho-1}F_{\chi_{\rho},Q}(n).    
\end{equation}
\end{definition}
\begin{remark}
Note that 
\[\Lambda_{Q}(n) = F_{\chi,Q}(n)\]
where $\chi$ is the character which is identically one and we set $\rho = 1$ as a ``zero'' (in reality of course it is a pole). We adopt this abuse throughout in order to not unnecessarily duplicate certain arguments. 
\end{remark}

The key result we will require is the following Fourier approximation result; this is \cite[Proposition~4.6]{Gre22}. This is crucially where our results rely on the use of zero-density estimates. 
\begin{proposition}\label{prop:cruc}
Suppose that $\sigma_0\le 1/9$ and $Q\le M^{\sigma_0}$. Then 
\[\sup_{\Theta\in \mbf{R}/\mbf{Z}}\bigg|\sum_{n\le M}\big(\Lambda(n) - \Lambda_{Q,\sigma_0}(n)\big)e(n\Theta)\bigg| \le M^{1+o(1)}Q^{-1/6}.\]
\end{proposition}

A key point in our analysis is that \cref{lem:L1} combines naturally with \cref{prop:cruc}; this is essentially Fourier inversion.
\begin{lemma}\label{lem:expan}
There exists a constant $C = C_{\ref{lem:expan}}\ge 1$ such that following holds. For $\sigma_0\le 1/9$ and $Q\le M^{\sigma_0}$, then 
\[\bigg|\sum_{n\le M}\big(\Lambda(n) - \Lambda_{Q,\sigma_0}(n)\big)e(n\Theta) \cdot \mbm{1}_{\mc{S}_{b,M}}(n)\bigg|\ll M^{1+o(1)}Q^{-1/6}(C\log g)^{k}.\]
\end{lemma}
\begin{proof}
Notice by \cref{prop:cruc,lem:L1} and Fourier inversion that
\begin{align*}
\bigg|\sum_{n\le M}\big(\Lambda(n) -& \Lambda_{Q,\sigma_0}(n)\big)e(n\Theta) \cdot \mbm{1}_{\mc{S}_{b,M}}(n)\bigg|\\ &=\bigg|\int_{\Theta'\in \mbf{R}/\mbf{Z}} \sum_{n\le M}\big(\Lambda(n) - \Lambda_{Q,\sigma_0}(n)\big)e(n(\Theta + \Theta'))\wh{\mbm{1}_{\mc{S}_{b,M}}}(\Theta')~d\Theta'\bigg|\\
&\le M^{1+o(1)}Q^{-1/6}\cdot\int_{\Theta'\in \mbf{R}/\mbf{Z}} \big|\wh{\mbm{1}_{\mc{S}_{b,M}}}(\Theta')\big|~d\Theta'\ll M^{1+o(1)}Q^{-1/6}(C\log g)^{k}.\qedhere
\end{align*}
\end{proof}

We finally require the following grand zero density estimate of Julita \cite{Jut77}; we remark that we do not require a log-free zero density estimate for our analysis. 
\begin{proposition}\label{prop:zero-dens}
Fix $\eps>0$. Uniformly for $\alpha \ge 4/5$, we have that 
\[\#\big\{\rho\in \Xi_{Q}(\alpha):\mf{R}\rho\ge \alpha\big\}\ll_{\eps} Q^{(6+\eps)(1-\alpha)}.\]
We have counted zeros with multiplicity; here the implied constants are computable.
\end{proposition}

We also require the following bounds for various terms in the expansion of $\Lambda_{Q,\sigma_0}$; these appear as \cite[Corollary~6.4]{Gre22} (after crudely bounding the series given via \cref{prop:zero-dens}). 
\begin{lemma}\label{lem:upper}
Fix $Q\le N$, $\eps>0$, and $\sigma_0\le 1/9$. We have that $|F_{\chi,Q}(n)|\ll \tau(n) (\log Q)^3$ and $|\Lambda_{Q,\sigma_0}(n)|\ll_\eps \tau(n)(\log Q)^{4}\max((Q^{(6+\eps)}/n)^{\sigma_0},1).$
\end{lemma}

We will finally require the following consequence of Siegel's theorem \cite{Sie35} and the Vinogradov--Korobov \cite{Vin58, Kor58} zero free region for Dirichlet $L$-functions. 

\begin{lemma}\label{lem:large-conduc}
Let $Q\le N$, $\sigma_0\le 1/2$, and $N$ is larger than an absolute constant. Then for $A\ge 1$ the following holds. Suppose that $\rho = \sigma + i\gamma \in \Xi_{Q}(\sigma_0)$ and $1 - \sigma \le (\log\log N)^2/\log N$. Then the modulus $q_{\rho}$ of the associated primitive Dirichlet character $\chi_\rho$ is $\gg_{A} (\log N)^{A}$. Furthermore there is at most $1$ primitive Dirichlet character $\chi_\rho$ corresponding to a zero in $\Xi_{A}$ such that the modulus is less than $\exp((\log N)/(\log\log N)^3)$.
\end{lemma}
\begin{proof}
We divide the proof into two cases, based on whether $|\gamma|$ is larger than or equal to, or less than or equal to $5$. If $|\gamma|\ge 5$, we have via the Vinogradov--Korobov zero free region for Dirichlet $L$-functions that the associated conductor is in fact larger than $\exp(\Omega(\log N/(\log\log N)^2))$ (see e.g., \cite{Kha24}). If $|\gamma|\le 5$, the desired bound follows from the Landau-Page theorem (e.g., \cite[Theorem~53,~Exercise~54]{Tao14}) applied over all Dirichlet character with conductor bounded by 
$$\exp(O(\log N/(\log\log N)^2))$$
modulo possibly the exception of a single real zero. The final real zero is then handled via Siegel's theorem (see \cite[Theorem~5.28]{IK04}).
\end{proof}

\section{Proof of \cref{thm:main} given \cref{lem:correction,lem:main-term}}

We are now in position to prove \cref{thm:main} modulo the crucial technical results of the paper. We first note that it suffices to consider the problem when $\Lambda$ is replaced by $\Lambda_{Q,\sigma_0}$. This is a completely standard consequence of \cref{lem:expan}.
\begin{lemma}\label{lem:3-AP}
Let $\sigma_0\le 1/9$ and $Q\le M^{\sigma_0}$. There exists a constant $C_{\ref{lem:3-AP}}\ge 1$ such that
\[\bigg|\sum_{x_1 + x_2 + x_3 = N}\bigg(\prod_{i=1}^{3}\Lambda(x_i)\mbm{1}_{\mc{S}_{b,M}}(x_i) -\prod_{i=1}^{3}\Lambda_{Q,\sigma_0}(x_i)\mbm{1}_{\mc{S}_{b,M}}(x_i)\bigg)\bigg|\le M^{2+o(1)}(C\log g)^{k}Q^{-1/6}.\]
\end{lemma}
\begin{proof}
Let $F(x_i) = \Lambda(x_i)\mbm{1}_{\mc{S}_{b,M}}(x_i)$ and $G(x_i) = \Lambda_{Q,\sigma_0}(x_i)\mbm{1}_{\mc{S}_{b,M}}(x_i)$. For functions $H_i:\mbf{Z}\to \mbf{C}$, we have that 
\[\sum_{x_1 + x_2 + x_3= N}\prod_{i=1}^{3}H_i(x_i) = \int_{\mbf{R}/\mbf{Z}}\prod_{i=1}^{3}\wh{H_i(\Theta)}e(-N\Theta)~d\Theta.\] Via telescoping, we have that 
\begin{align*}
&\bigg|\sum_{x_1 + x_2 + x_3 = N}\bigg(\prod_{i=1}^{3}\Lambda(x_i)\mbm{1}_{\mc{S}_{b,M}}(x_i) -\prod_{i=1}^{3}\Lambda_{Q,\sigma_0}(x_i)\mbm{1}_{\mc{S}_{b,M}}(x_i)\bigg)\bigg|\\
&=\bigg|\sum_{x_1 + x_2 + x_3 = N}(F-G)(x_1)F(x_2)F(x_3) +G(x_1)(F-G)(x_2)F(x_3) + G(x_1)G(x_2)(F-G)(x_3)\bigg|
\\
&\qquad\le\frac{3}{2}\cdot \int_{\mbf{R}/\mbf{Z}}|\wh{F}(\Theta-\wh{G}(\Theta)|\cdot (|\wh{F}(\Theta)|^2+|\wh{G}(\Theta)|^2)~d\Theta \\
&\qquad\le \sup_{\Theta\in \mbf{R}/\mbf{Z}}|\wh{F}(\Theta)-\wh{G}(\Theta)| \cdot \frac{3}{2}\cdot \int_{\mbf{R}/\mbf{Z}}|\wh{F}(\Theta)|^2+|\wh{G}(\Theta)|^2~d\Theta \ll (M^{1+o(1)}Q^{-1/6}(C\log g)^k) \cdot M^{1+o(1)}
\end{align*}
where in the final line we have applied \cref{lem:expan} and Parseval. (The $L^2$-bound for $F$ is immediate as $\Lambda(n)\le \log n$ pointwise and $G$ follows via \cref{lem:upper}.)
\end{proof}

We now remove the contribution from all terms which correspond to zeros with sufficiently small real part. For the remainder of the proof, set
\[\sigma_0 = 10^{-3}\text{ and }Q = M^{\sigma_0}.\]
We abbreviate 
\[\Xi := \Xi_{Q}(\sigma_0)\]
and define 
\[\Xi_{A}:= \bigg\{\rho = \sigma + it: \rho\in \Xi,~1-\sigma \le\frac{A\log\log M}{\log M}\bigg\}.\]
As in $\Xi$, zeros in $\Xi_{A}$ are treated with multiplicity. We furthermore define 
\[\Lambda_{A,\on{Approx}}(x) = \Lambda_{Q}(x) - \sum_{\rho\in \Xi_{A}}x^{\rho-1}F_{\chi_{\rho},Q}(x).\]
The following states that it suffices to consider $\Lambda_{A,\on{Approx}}$; this is a consequence of \cref{lem:divisor-function} and \cref{prop:zero-dens}.
\begin{lemma}\label{lem:approx-remov}
There exists a constant $C = C_{\ref{lem:approx-remov}}\ge 1$ such that if $g\ge C$ and $A\ge 1$, then 
\[\bigg|\sum_{\substack{x_1 + x_2 + x_3 = N\\x_i\in \mc{S}_{b}}}\bigg(\prod_{i=1}^{3}\Lambda_{Q,\sigma_0}(x_i) - \prod_{i=1}^{3}\Lambda_{A,\on{Approx}}(x_i)\bigg)\bigg|\ll_{A}(\log M)^{-A/18 + O(1)} \cdot  \sum_{\substack{x_1 + x_2 + x_3 = N\\x_i\in \mc{S}_{b}}} 1.\]
\end{lemma}
\begin{proof}
Via \cref{prop:zero-dens} and \cref{lem:upper}, we see that we have the bound 
\[\big|\Lambda_{Q,\sigma_0}(x) - \Lambda_{A,\on{Approx}}(x)\big|\ll \tau(x) \cdot (\log M)^{3}\cdot \max\bigg((Q^8/x)^{\sigma_0},(Q^8/x)^{\frac{A\log\log M}{\log M}}\bigg).\]
We therefore have that 
\begin{align*}
\bigg|\sum_{\substack{x_1 + x_2 + x_3 = N\\x_i\in \mc{S}_{b}}}\bigg(\prod_{i=1}^{3}&\Lambda_{Q,\sigma_0}(x_i) - \prod_{i=1}^{3}\Lambda_{A, \mathrm{Approximate}}(x_i)\bigg)\bigg|\\
&\ll M^{1 + 17/18 + o(1)} + \sum_{\substack{x_1 + x_2 + x_3 = N\\\min(x_i)\ge M^{17/18}\\x_i\in \mc{S}_{b}}}(\log M)^{-A/18}\prod_{i=1}^{3}\tau(x_i)\\
&\ll M^{35/18 + o(1)} + (\log M)^{-A/18} \sum_{\substack{x_1 + x_2 + x_3 = N\\x_i\in \mc{S}_{b}}}\prod_{i=1}^{3}\tau(x_i)^{3}\\
&\ll_{A} (\log M)^{-A/18 + O(1)} \sum_{\substack{x_1 + x_2 + x_3 = N\\x_i\in \mc{S}_{b}}} 1;
\end{align*}
here we have used \cref{lem:sol-bound} to remove the first term.
\end{proof}

The remainder of the analysis will be devoted to computing various terms when expanding into $\Lambda_{A}(x_1)$. These will be proven in the next section and are the technical heart of the remainder of the paper. The most delicate part of the analysis is ruling out various terms arising from various ``correcting'' zeros. 
\begin{lemma}\label{lem:correction}
Fix $A,B\ge 1$. We have that if $\chi_1$ and $\chi_2$ are not both trivial and the corresponding zeros lie in $\Xi_{A}$, then
\[\sum_{\substack{x_3\in \mc{S}_b}}\tau(x_3)\bigg|\sum_{\substack{x_1 + x_2 = N-x_3\\x_i\in \mc{S}_b}}\prod_{j=1}^{2}x_{j}^{\rho_j - 1}F_{\chi_j,Q}(x_j)\bigg|\ll_{A,B}(\log M)^{-B}\sum_{\substack{x_1 + x_2 +x_3= N\\x_i\in \mc{S}_b}}1.\]
\end{lemma}

We also now compute the main term; this computation is nontrivial as the main term oscillates as $N$ varies. 
\begin{lemma}\label{lem:main-term}
Fix $A\ge 1$. We have that 
\[\sum_{\substack{x_1+x_2+x_3 = N\\x_i\in \mc{S}_b}}\prod_{i=1}^{3}\Lambda_{Q}(x_i) = \big(1 + O_A((\log M)^{-A})\big)\prod_{\substack{p\nmid g\\p|N}}\bigg(1 - \frac{1}{(p-1)^2}\bigg)\prod_{\substack{p\nmid g\\p\nmid N}}\bigg(1 + \frac{1}{(p-1)^3}\bigg)\sum_{\substack{x_1+x_2+x_3 = N\\x_i\in \mc{S}_b\\(x_i,g) = 1}}1. \]    
\end{lemma}

We now complete the proof of \cref{thm:main}.
\begin{proof}[{Proof of \cref{thm:main}}]
Since $Q$ is sufficiently large, \cref{lem:sol-bound} and \cref{lem:3-AP} apply, and it suffices to estimate 
\[\sum_{\substack{x_1 + x_2 + x_3 = N\\x_i\in \mc{S}_{b,M}}}\prod_{i=1}^{3}\Lambda_{Q,\sigma_0}(x_i).\]
Note that for $g$ sufficiently large we have that 
\[\{x_1 + x_2 + x_3 \equiv N\imod g:(x_i,g) = 1\wedge x_i\neq b\}\ge 1.\]
Combining with \cref{lem:init}, we have that for $N$ sufficiently large that
\[\sum_{\substack{x_1 + x_2 + x_3 = N\\x_i\in \mc{S}_{b,M}\\(x_i,g) = 1}}1\gg_{g} \sum_{\substack{x_1 + x_2 + x_3 = N\\x_i\in \mc{S}_{b,M}}}1.\]
By \cref{lem:approx-remov}, we then have that it suffices to consider $\Lambda_{A,\on{Approx}}$ in place of $\Lambda_{Q,\sigma_0}$. We now expand out all the terms in $\Lambda_{A,\on{Approx}}$; there are $(\log M)^{O(A)}$ such terms by \cref{prop:zero-dens}. For any term involving a zero $x^{\rho - 1}F_{\chi_{\rho},Q}(x)$ for any of $x_1$, $x_2$, and $x_3$ we may apply \cref{lem:correction} (with $B$ sufficiently large with respect to $A$ and possibly after permuting the variables) and \cref{lem:upper} to remove these terms. (Handling terms where one of $x_i$ is smaller than $Q^{6+\eps}$ and thus not technically divisor bounded is routine as in \cref{lem:approx-remov}.) Therefore it suffices to estimate 
\[\sum_{\substack{x_1 + x_2 + x_3 = N\\x_i\in \mc{S}_{b,M}}}\prod_{i=1}^{3}\Lambda_{Q}(x_i)\]
which is exactly handled by \cref{lem:main-term}.
\end{proof}

\section{Proof of \texorpdfstring{\cref{lem:correction,lem:main-term}}{Lemmas 5.3 and 5.4}}
\subsection{Handling correction terms}
We first require a pair of results regarding the divisibility properties of $x_3$; these follow via inspecting the lowest order bits of $x_3$.
\begin{lemma}\label{lem:divis}
There exists an absolute constant $C = C_{\ref{lem:divis}}\ge 1$ such that the following hold. Fix an integer $d\le M$ and $g\ge C_{\ref{lem:divis}}$. Then
\[\sum_{\substack{x_1 + x_2 + x_3 = N\\x_i\in \mc{S}_b}}\mbm{1}[d|x_1+x_2]\ll_{g} d^{-1/C} \sum_{\substack{x_1 + x_2 + x_3 = N\\x_i\in \mc{S}_b}}1.\]
\end{lemma}
\begin{proof}
The desired condition is equivalent to $x_3 \equiv N\imod d$. Let $L = \lfloor \log d/(4\log g)\rfloor$; we have that $L\le k/4$. Via \cref{lem:chernoff} and \cref{lem:chernoff-hoeff}, we find that 
\[\mb{P}_{\substack{x_1 + x_2 + x_3 = N\\x_i\in \mc{S}_b}}\bigg[\sum_{j=0}^{L}\mbm{1}\big[\min(x_{2,j} + x_{3,j}, 2g - x_{2,j}-x_{3,j})\le g^{1/2}\big]\ge L/2\bigg]\ll_{g} \exp\big(-\Omega(L\cdot \log g)\big) = d^{-\Omega(1)}.\]
We now consider $x_{3}$ having revealed $x_{3,j}$ for $j>L$ and $x_{2,j} + x_{3,j}$ for $j\le L$ under the assumption that 
\[\sum_{j=0}^{L}\mbm{1}\big[\min(x_{2,j} + x_{3,j}, 2g - x_{2,j}-x_{3,j})\le g^{1/2}\big]\le L/2;\]
the complementary event contributes negligibly. Note that there are $g^{\Omega(L)}$ choices for $(x_{3,j})_{0\le j\le L}$ and there is at most $1$ choice which guarantees $x_3\equiv N\imod d$. The result follows. 
\end{proof}

The second lemma provides a variant of \cref{lem:divis} which is ultimately used to guarantee various divisibility constraints on $x_1 + x_2$. 
\begin{lemma}\label{lem:divis-small}
There exists a constant $C = C_{\ref{lem:divis-small}}$ such that the following holds. Fix $g\ge C$, $A\ge 1$, and $d\le M$. Then 
\[\sum_{\substack{x_1 + x_2 + x_3 = N\\x_i\in \mc{S}_b}}\mbm{1}[d/(x_1+x_2,d)\le d^{1/2}]\ll_{g}  d^{-1/C}\sum_{\substack{x_1 + x_2 + x_3 = N\\x_i\in \mc{S}_b}}1.\]
\end{lemma}
\begin{proof}
Let $(x_1+x_2,d) = d'$ and note that there are only $\tau(d) = d^{o(1)}$ choices for $d'$. Noting that $d'\ge d^{1/2}$ the conclusion follows from \cref{lem:divis} (assuming $C_{\ref{lem:divis-small}}$ is sufficiently large).
\end{proof}

As outlined in the introduction, we only extract cancellation after fixing the sum $x_1 + x_2$. In order to proceed efficiently, this necessitates extracting a certain additional character sum from the characters present in \cref{def:aux-func}. We now collect the necessary bounds; they are proven in \cref{app:char-sum} and are either elementary or follow from the Weil bound. It suffices to prove any power-type cancellation bounds when $|a_2-a_1|$ is sufficiently small with respect to $\max(a_1,a_2)$ for our application.

\begin{lemma}\label{lem:char-sum}
Let $p$ be prime, $a_1\ge a_2$, and $\chi_i$ be primitive Dirichlet characters with moduli $p^{a_i}$. Define
\[W(b_1,b_2,T) = p^{-a_2}\sum_{t\in \mbf{Z}/p^{a_2}\mbf{Z}}\chi_1(p^{a_1-a_2}t + b_1)\chi_2(t + b_2)e\bigg(\frac{Tt}{p^{a_2}}\bigg).\]

If $a_1> 1$, $\alpha = \lfloor a_1/2\rfloor$, and $\alpha + 2 > 2(a_1-a_2) + v_p(T)$, then 
\[\sup_{b_1,b_2}|W(b_1,b_2,T)|\le \min\big(1, 32 \cdot p^{-\lceil (\alpha - v_p(T))/2\rceil + (a_1-a_2)}\big).\]
Furthermore if $(a_1,a_2) = (1,1)$ and $p\nmid T$, then 
\[\sup_{b_1,b_2}|W(b_1,b_2,T)|\le 4p^{-1/2}.\]
\end{lemma}

We next require the following elementary fact regarding the subtracting fractions with fixed denominators.
\begin{lemma}\label{lem:frac}
Let $p$ be prime and $a_1\le a_2$. Then 
\[\sup_{t\in \mbf{Z}}\bigg|\bigg\{\frac{b_1}{p^{a_1}} + \frac{b_2}{p^{a_2}} \equiv \frac{t}{p^{a_2}}\imod 1: b_i\in \mbf{Z}/p^{a_i}\mbf{Z}\bigg\}\bigg|\le p^{a_1}.\]
Furthermore if $a_1<a_2$, then 
\[\sup_{\substack{t\in \mbf{Z}\\(t,p) \neq 1}}\bigg|\bigg\{\frac{b_1}{p^{a_1}} + \frac{b_2}{p^{a_2}} \equiv \frac{t}{p^{a_2}}\imod 1: b_i\in (\mbf{Z}/p^{a_i}\mbf{Z})^{\ast}\bigg\}\bigg|=0.\]
\end{lemma}
\begin{proof}
For the first part, 
\[\frac{b_1}{p^{a_1}} + \frac{b_2}{p^{a_2}} = \frac{p^{a_2-a_1}b_1 + b_2}{p^{a_2}}\]
and therefore, 
\[t\equiv b_2 \imod p^{a_2-a_1}\text{ and }b_1 \equiv (t-b_2)p^{a_1-a_2}\imod p^{a_1}.\]
Thus, given $t$ and $b_1$, the first relation determines the last $(a_2-a_1)$ digits of $b_2$ in base $p$ while the second determines the first $a_1$ as desired. For the second part, note that if $a_2>a_1$, then
\[(p^{a_2-a_1}b_1 + b_2, p) = (b_2,p) = 1\]
and the result follows.
\end{proof}

We finally require the (technical) notion of a ``well-conditioned'' product measures.

\begin{definition}\label{def:well-conditioned}
A sequence $\mc{B} = (\mc{B}_0,\ldots,\mc{B}_{k-1})$ is $C$-well conditioned if:
\begin{itemize}
    \item For all intervals $\mc{I}\subseteq [k/4]$ of length $C\log\log N/\log g$, we have 
    \[\sum_{i\in \mc{I}}\mbm{1}[|\mc{B}_i|\ge g^{99/100}]\ge 99|\mc{I}|/100. \]
    \item $\mc{B}_i$ are intervals missing at most $2$ points
    \item For all $1\le a<b$ with $(b,g) = 1$ and $b\le \exp((\log\log N)^5)$, we have that 
    \[\sum_{\substack{j\in [k/8,k/4]\\|\mc{B}_j|\ge 4}}\mbm{1}[\snorm{a \cdot g^j/b}_{\mbf{R}/\mbf{Z}}\ge 1/g^2]\ge \frac{k}{40\log b}.\]
\end{itemize}
\end{definition}
\begin{remark}
The first and second conditions are used to guarantee that various ``large--sieve'' estimates are effective at all scales. The second and third conditions are used to prove various Fourier coefficients of sufficiently small denominator are ``tiny.''
\end{remark}

We now derive the associated consequences of being well-conditioned. We remark that the second estimate is closely related to \cite[Lemma~5.4]{May22}.
\begin{lemma}\label{lem:well-output}
There exists $C_{\ref{lem:well-output}}\ge 1$ such that the following holds. Let $g,C\ge C_{\ref{lem:well-output}}$, $\mc{B}$ be $C$-well conditioned, and $d\le Q\le g^{k/10}$. Then we have the large sieve estimate
\[\sup_{\beta\in \mbf{R}}\sum_{\substack{b\sim Q\\d|b}}\sum_{\substack{0<a<b\\(a,b) = 1}}\bigg|\wh{\mu_{\vec{\mc{B}}}}\bigg(\frac{a}{b}+ \beta\bigg)\bigg|\ll (\log N)^{O(C)} \cdot \prod_{j=0}^{k-1}|\mc{B}_j| \cdot \bigg(\frac{Q^2}{d}\bigg)^{1/50}.\]
Furthermore if $b = d_1d_2$ with $d_1|g^{k/10}$, $(d_2,g) = (a,b) =1$, $d_2\le \exp((\log k)^{4})$, and $d_2\neq 1$, then we have the $L^\infty$ estimate
\[\bigg|\wh{\mu_{\vec{\mc{B}}}}\bigg(\frac{a}{b}\bigg)\bigg| \le \prod_{j=0}^{k-1}|\mc{B}_j| \cdot \exp\bigg(\frac{-ck}{g^{5}\log d_2}\bigg)\]
for some $c > 0$.
\end{lemma}
\begin{proof}
We prove each of the parts in turn. For the first part, if $Q^2/d\le (\log N)^{O(C)}$, the resulting estimate is trivial and thus we may assume that $Q^2/d\ge (\log N)^{O(C)}$. Take $L = \lfloor \log(Q^2/d)/\log g\rfloor$; note that $L \le k/5$. We have from \cref{lem:large-sieve} that 
\[\sup_{\beta\in \mbf{R}}\sum_{\substack{b\sim Q\\d|b}}\sum_{\substack{0<a<b\\(a,b) = 1}}\bigg|\wh{\mu_{\vec{\mc{B}}}}\bigg(\frac{a}{b}+ \beta\bigg)\bigg|\ll \prod_{j>L}|\mc{B}_j| \cdot \bigg(\frac{Q^2}{d}\cdot (C_{\ref{lem:large-sieve}}\log g)^{L} + (C_{\ref{lem:large-sieve}}g\log g)^{L}\bigg).\]
Noting that at least $99/100$-fraction of the sets for $j\le L$ satisfy $|\mc{B}_j|\ge g^{99/100}$, the desired bound follows for $g$ sufficiently large. 

For the second part, we only assume the third condition in $C$-well conditioned. Via partitioning further we may therefore assume that $|\mc{B}_i| = 1$ for $i\notin [k/8,k/4)$. Given $x\sim \mu_{\mc{B}}$ we may write $x = \wt{x} + g^{k/8}x'$ where $x'\sim \mu_{\mc{B}'}$ where $\mc{B}' = (\mc{B}_{k/8},\ldots,\mc{B}_{k/4})$. It suffices to prove that 
\[\sup_{(a,d_2) = 1}\bigg|\wh{\mu_{\vec{\mc{B}'}}}\bigg(\frac{a}{d_2}\bigg)\bigg| \le \prod_{j=k/8}^{k/4}|\mc{B}_j| \cdot \exp\bigg(\frac{-ck}{g^5\log d_2}\bigg).\]
Note that if $|\mc{B}_j|\ge 4$ and is an interval missing at most two elements, $\mc{B}_j$ contains at least two consecutive elements. Thus via the triangle inequality, 
\[|\mc{B}_j|^{-1} \cdot \bigg|\sum_{x_j\in \mc{B}_j}e\bigg(\frac{ax_j\cdot g^{j}}{d_2}\bigg)\bigg|\le \frac{|\mc{B}_j|-2 + \big|1 + e\big(\frac{a\cdot g^{j}}{d_2}\big)\big|}{|\mc{B}_j|}\le 1-\frac{2-\big|1 + e\big(\frac{a\cdot g^{j}}{d_2}\big)\big|}{g}.\]
Note that for at least $k/(40\log d_2)$ indices,
\[\frac{2-\big|1 + e\big(\frac{a\cdot g^{j}}{d_2}\big)\big|}{g}\ge \frac{1}{g^{3}}.\]
Therefore,
\[\prod_{j=k/8}^{k/4}|\mc{B}_j|^{-1} \cdot \sup_{(a,d_2) = 1}\bigg|\wh{\mu_{\vec{\mc{B}'}}}\bigg(\frac{a}{d_2}\bigg)\bigg| = \prod_{j=k/8}^{k/4}\bigg(|\mc{B}_j|^{-1} \cdot \bigg|\sum_{x_j\in \mc{B}_j}e\bigg(\frac{ax_j\cdot g^{j}}{d_2}\bigg)\bigg|\bigg)\le \bigg(1-\frac{1}{g^3}\bigg)^{\Omega(k/\log d_2)}\]
as desired; noting that $g$ is sufficiently large gives the stated bound. 
\end{proof}

We now state the main lemma in our analysis. The deduction of \cref{lem:correction} then reduces to various probabilistic guarantees on the measure $\mu_{\vec{\mc{B}}}$ and divisibility properties of $x_1 + x_2$.
\begin{lemma}\label{lem:key-computation}
Fix $A,B\ge 1$. There exists $C = C_{\ref{lem:key-computation}}\ge 1$ and $C' = C_{\ref{lem:key-computation}}'(A,B,g)\ge g^2(A+B)$ such that the following holds. 

Let $g\ge C$. Furthermore, let $\chi_1$ and $\chi_2$ be primitive Dirichlet characters such that the corresponding zeros $\rho_i\in \Sigma_{A}$, $q_i$ are the corresponding moduli, and $\chi_2$ is not trivial.

Take $t_i\in \{0,\ldots,2(g-1)\}$ for $i\in [k]$ and define $T = \sum_{i=0}^{k-1}t_ig^{i}$. Suppose that $T$ satisfies:
\begin{align*}
\prod_{p\le (\log N)^{C'}} p^{v_p(T)} &\le \exp((\log\log N)^{5/2})\\
p|(q_1q_2,T)&\implies p\le (\log N)^{C'}\\
\sup_p v_p(T)&\le C'(\log\log N)\\
\log(q_2/(T,q_2))&\ge C'^2 \log\log N.
\end{align*}

Let $x_1\sim \mu_{\vec{\mc{B}}}$ with $\vec{\mc{B}}$ where $\mc{B}$ is $C'$ well-conditioned and $|\mc{B}_j| = 1$ for $j\in [k/2,k]$. Furthermore, suppose there exists $j_0\in [2k/3,3k/4]$ such that $\mc{B}_{j_0}\cap \{0,t_{j_0}-1,t_{j_0}\} = \emptyset$. 

Then 
\[\bigg|\sum_{\substack{x_1 + x_2 = T\\x_1\sim \mu_{\vec{\mc{B}}}}}\prod_{j=1}^{2}x_{j}^{\rho_j - 1}F_{\chi_j,Q}(x_j)\bigg|\ll_{A,B,g}(\log M)^{-B}\cdot \prod_{i\in [k]}|\mc{B}_i|.\]
\end{lemma}
\begin{proof}
As there exists $j_0 \in [2k/3,3k/4]$ such that $\mc{B}_{j_0}\cap \{0,t_{j_0}-1,t_{j_0}\} = \emptyset$, we deterministically have that $|x_1|,|x_2|\ge M^{2/3}$ for all choices in the support of the sum. Furthermore as $|\mc{B}_i| = 1$ for $i\ge k/2$, there exists integers $\wt{x_i}$ such that $|x_i - \wt{x_i}|\ll_{g} M^{1/2}$ for all choices of $x_i$ in the support of the sum. Since $\sigma_0 = 10^{-10}$, this implies that 
\[\bigg|\prod_{j=1}^{2}x_{j}^{\rho_j - 1} - \prod_{j=1}^{2}\wt{x_{j}}^{\rho_j - 1}\bigg|\ll \frac{|\mf{Im}(\rho_1)| + |\mf{Im}(\rho_2)|}{M^{1/6 + o(1)}}\ll M^{-1/8}.\]

Recall that $F_{\chi_j,Q}$ is defined in \cref{eq:ApproximantError}. As $|F_{\chi_j,Q}(x_j)|\le M^{o(1)}$ by \cref{lem:upper} and the divisor bound, it suffices to prove that 
\[\prod_{j=1}^{2}\big|\wt{x_{j}}^{\rho_j - 1}\big|\cdot \bigg|\sum_{\substack{x_1 + x_2 = T\\x_1\sim \mu_{\vec{\mc{B}}}}}\prod_{j=1}^{2}F_{\chi_j,Q}(x_j)\bigg|\ll_{A,B}(\log M)^{-B}\cdot \prod_{i\in [k]}|\mc{B}_i|.\]

Noting that $\prod_{j=1}^{2}\big|\wt{x_{j}}^{\rho_j - 1}\big|\le 1$, it suffices to prove 
\[\bigg|\sum_{\substack{x_1 + x_2 = T\\x_1\sim \mu_{\vec{\mc{B}}}}}\prod_{j=1}^{2}F_{\chi_j,Q}(x_j)\bigg|\ll_{A,B}(\log M)^{-B}\cdot \prod_{i\in [k]}|\mc{B}_i|.\]
\textbf{Step 1: Expanding the sum and extracting \cref{lem:char-sum}.}

We separate prime factors of $q_i$ into a series of classes. Define
\begin{align*}
q_i &:= \prod_{v_p(q_i)/v_p(q_{3-i})\in [9/10,10/9]}p^{v_p(q_i)}\prod_{\substack{v_p(q_i)/v_p(q_{3-i})\notin [9/10,10/9]\\\min(v_p(q_1),v_p(q_2))>0}}p^{v_p(q_i)}\prod_{v_p(q_i)>v_p(q_{3-i}) = 0}p^{v_p(q_{i})}:= \wh{q_i}\wt{q_i}q_i^{\ast}.
\end{align*}
Furthermore, define 
\[\wh{q} = \on{lcm}(\wh{q_1},\wh{q_2})~\text{ and }\wt{q} = \on{lcm}(\wt{q_1},\wt{q_2}).\]
Notice that
\begin{align*}
&\sum_{\substack{x_1 + x_2 = T\\x_1\sim \mu_{\vec{\mc{B}}}}}\prod_{j=1}^{2}F_{\chi_j,Q}(x_j) = \sum_{\substack{x_1 + x_2 = T\\x_1\sim \mu_{\vec{\mc{B}}}}}\prod_{i=1}^{2}\ol{\tau(\chi_i)}\sum_{r_1,r_2\le Q}\frac{\prod_{i=1}^{2}\mu(r_i)\ol{\chi_i(r_i)}}{\prod_{i=1}^{2}\phi(q_ir_i)}\sum_{\substack{b_i\in \mbf{Z}/(q_ir_i\mbf{Z})\\(b_i,r_i) = 1}}\prod_{i=1}^{2}\chi_i(b_i)e\bigg(\sum_{i=1}^{2}\frac{b_ix_i}{q_ir_i}\bigg)\\
&= \sum_{\substack{x_1 + x_2 = T\\x_1\sim \mu_{\vec{\mc{B}}}}}\prod_{i=1}^{2}\ol{\tau(\chi_i)}\sum_{\substack{r_1,r_2\le Q}}\frac{\prod_{i=1}^{2}\mu(r_i)\ol{\chi_i(r_i)}}{\prod_{i=1}^{2}\phi(q_ir_i)} \\
&\qquad\qquad\qquad\qquad\qquad\cdot \bigg((\wh{q_1},\wh{q_2})^{-1}\sum_{\substack{b_i\in \mbf{Z}/(q_ir_i\mbf{Z})\\t\in \mbf{Z}/(\wh{q_1},\wh{q_2})\mbf{Z}\\(b_i,r_i) = 1}}\prod_{i=1}^{2}\chi_i\bigg(b_i + \frac{tq_ir_i}{(\wh{q_1},\wh{q_2})}\bigg)e\bigg(\sum_{i=1}^{2}\frac{b_ix_i}{q_ir_i}+\sum_{i=1}^{2}\frac{x_it}{(\wh{q_1},\wh{q_2})}\bigg)\bigg) \\
&= \sum_{\substack{x_1 + x_2 = T\\x_1\sim \mu_{\vec{\mc{B}}}}}\prod_{i=1}^{2}\ol{\tau(\chi_i)}\sum_{\substack{r_1,r_2\le Q}}\frac{\prod_{i=1}^{2}\mu(r_i)\ol{\chi_i(r_i)}}{\prod_{i=1}^{2}\phi(q_ir_i)} \\
&\qquad\qquad\qquad\qquad\qquad\cdot \bigg((\wh{q_1},\wh{q_2})^{-1}\sum_{\substack{b_i\in \mbf{Z}/(q_ir_i\mbf{Z})\\t\in \mbf{Z}/(\wh{q_1},\wh{q_2})\mbf{Z}\\(b_i,r_i) = 1}}\prod_{i=1}^{2}\chi_i\bigg(b_i + \frac{tq_ir_i}{(\wh{q_1},\wh{q_2})}\bigg)e\bigg(\sum_{i=1}^{2}\frac{b_ix_i}{q_ir_i}\bigg)e\bigg(\frac{Tt}{(\wh{q_1},\wh{q_2})}\bigg)\bigg) \\
&= \prod_{i=1}^{2}\ol{\tau(\chi_i)}\sum_{\substack{r_1,r_2\le Q\\(r_i,q_i) =1}}\frac{\prod_{i=1}^{2}\mu(r_i)\ol{\chi(r_i)}}{\prod_{i=1}^{2}\phi(q_i)\phi(r_i)}\sum_{\substack{b_i\in \mbf{Z}/(q_ir_i\mbf{Z})\\(b_i,r_i) = 1}}\bigg(\sum_{\substack{x_1 + x_2 = T\\x_1\sim \mu_{\vec{\mc{B}}}}}e\bigg(\sum_{i=1}^{2}\frac{b_ix_i}{q_ir_i} \bigg)\bigg)\\
&\qquad\qquad\qquad\qquad\qquad\cdot \bigg((\wh{q_1},\wh{q_2})^{-1}\sum_{t\in \mbf{Z}/(\wh{q_1},\wh{q_2})\mbf{Z}}\prod_{i=1}^{2}\chi_i\bigg(b_i + \frac{tq_ir_i}{(\wh{q_1},\wh{q_2})}\bigg)e\bigg(\frac{Tt}{(\wh{q_1},\wh{q_2})}\bigg)\bigg).
\end{align*}
The key point is that the final sum is exactly of the form in \cref{lem:char-sum} (after factoring into different prime factors). Furthermore, note that the internal sum is $0$ unless $(b_i,\wt{q_i}q_i^{\ast}) = 1$ as otherwise the argument in $\chi_i$ vanishes over the whole sum. 

\textbf{Step 2: Bounding inner sums via \cref{lem:char-sum}.}

We now apply absolute values brutally and estimate the desired sum.
We may write 
\[\chi_i(t) = \prod_{p|q_i}\chi_{i,p}(t)\]
where $\chi_{i,p}$ is a character modulo $p^{v_p(q_i)}$. Using \cite[Lemma~C.1]{Gre22} we have that each $\chi_{i,p}$ is primitive and using the Chinese remainder theorem, we may obtain a factorization of the character sum of 
\[\bigg((\wh{q_1},\wh{q_2})^{-1}\sum_{t\in \mbf{Z}/(\wh{q_1},\wh{q_2})\mbf{Z}}\prod_{i=1}^{2}\chi_i\bigg(b_i + \frac{tq_ir_i}{(\wh{q_1},\wh{q_2})}\bigg)e\bigg(\frac{Tt}{(\wh{q_1},\wh{q_2})}\bigg)\bigg)\]
based on the prime factors of $(\wh{q_1},\wh{q_2})$ (see e.g., \cite[(12.21)]{IK04}). Then applying \cref{lem:char-sum}, we have that 
\begin{align*}
\sup_{b_1,b_2}&\bigg|(\wh{q_1},\wh{q_2})^{-1}\sum_{t\in \mbf{Z}/(\wh{q_1},\wh{q_2})\mbf{Z}}\prod_{i=1}^{2}\chi_i\bigg(b_i + \frac{tq_ir_i}{(\wh{q_1},\wh{q_2})}\bigg)e\bigg(\frac{Tt}{(\wh{q_1},\wh{q_2})}\bigg)\bigg|\\
&\ll \prod_{\substack{p|(\wh{q_1},\wh{q_2})\\v_p((\wh{q_1},\wh{q_2}))=1}}\min\big(1, 4p^{v_p(T)} \cdot p^{-1/4}\big)\cdot \prod_{\substack{p|(\wh{q_1},\wh{q_2})\\v_p((\wh{q_1},\wh{q_2}))\ge 2}}\min\big(1, 32p^{v_p(T)} \cdot p^{-\lceil \lfloor v_p(\wh{q})/2\rfloor/2 \rceil + v_p(\wh{q}) - v_p((\wh{q_1},\wh{q_2}))}\big)\\
&\ll_{g} \min\bigg(\frac{\exp(O((\log\log N)^{5/2}))}{(\wh{q_1},\wh{q_2})^{1/8}},\frac{(\log N)^{gC'}}{\prod_{p|g}p^{v_p((\wh{q_1},\wh{q_2}))/8}}, 1\bigg)\\
&\ll_{g} \min\bigg(\frac{\exp(O((\log\log N)^{5/2}))}{\wh{q}^{1/9}}, \frac{(\log N)^{gC'}}{\prod_{p|g}p^{v_p((\wh{q_1},\wh{q_2}))/8}}, 1\bigg)\\
&:=W.
\end{align*}
Here we are using that 
\[\prod_{\substack{p\le (\log N)^{C'}}}p^{v_p(T)}\le \exp((\log\log N)^{5/2})\]
and also that $v_p(T) \le C'\log\log N$.

\textbf{Step 3: Applying \cref{lem:frac}}
For the sake of notational simplicity, let 
\[F(\Theta) = \Big|\sum_{x_1\sim \mu_{\vec{\mc{B}}}}e(\Theta x_1)\Big| \cdot \prod_{j}|\mc{B}_j|^{-1}.\]

Since $|\tau(\chi_i)| \le q_i^{1/2}$, it suffices after negating $b_2$ to estimate 
\begin{align*}
&W\sum_{\substack{r_1,r_2\le Q\\(r_i,q_i) =1\\\mu(r_1),\mu(r_2) \neq 0}}\frac{(q_1q_2)^{1/2}}{\prod_{i=1}^{2}\phi(q_i)\phi(r_i)}\sum_{\substack{b_i\in \mbf{Z}/(q_ir_i\mbf{Z})\\(b_i,r_i) = (b_i,\wt{q_i}q_i^{\ast})= 1}}F\bigg(\frac{b_1}{q_1r_1}+\frac{b_2}{q_2r_2}\bigg)\\
&\ll W(\log\log N)^{O(1)}\sum_{\substack{r_1,r_2\le Q\\(r_i,q_i) =1\\\mu(r_1) , \mu(r_2) \neq 0}}\frac{1}{(q_1q_2)^{1/2}(r_1r_2)}\sum_{\substack{b_i\in \mbf{Z}/(q_ir_i\mbf{Z})\\(b_i,r_i) = (b_i,\wt{q_i}q_i^{\ast})= 1}}F\bigg(\frac{b_1}{q_1r_1}+\frac{b_2}{q_2r_2}\bigg)
\end{align*}
and prove it is $\ll_{B}(\log N)^{-B}$. We now write $(r_1,r_2) = r$ and have $r_1 = rr_1^{\ast}$ and $r_2 = rr_2^{\ast}$. As $\mu(r_1) = \mu(r_2)\neq 0$, we have that $(r,r_i^{\ast}) = 1$. The above is bounded by
\begin{align*}
W(\log\log N)^{O(1)}\sum_{\substack{r,r_1^{\ast},r_2^{\ast}\le Q\\(rr_i^{\ast},q_i) =1\\\mu(r_1^\ast), \mu(r_2^{\ast}),\mu(r) \neq 0 \\ (r,r_i^{\ast}) = (r_1^{\ast},r_2^{\ast}) = 1}}\frac{1}{(q_1q_2)^{1/2}(r^2r_1^{\ast}r_2^{\ast})}\sum_{\substack{b_i\in \mbf{Z}/(q_irr_i^{\ast}\mbf{Z})\\(b_i,rr_i^{\ast}) = (b_i,\wt{q_i}q_i^{\ast})= 1}}F\bigg(\frac{b_1}{q_1rr_1^{\ast}}+\frac{b_2}{q_2rr_2^{\ast}}\bigg).
\end{align*}
We next use that $(rr_i^{\ast},q_i) = 1$ in order to split the final sum. The above may be written as 
\begin{align*}
W(\log\log N)^{O(1)}\sum_{\substack{r,r_1^{\ast},r_2^{\ast}\le Q\\(rr_i^{\ast},q_i) =1\\\mu(r_1^\ast), \mu(r_2^{\ast}),\mu(r) \neq 0\\ (r,r_i^{\ast}) = (r_1^{\ast},r_2^{\ast}) = 1}}\frac{1}{(q_1q_2)^{1/2}(r^2r_1^{\ast}r_2^{\ast})}\sum_{\substack{b_i\in \mbf{Z}/q_i\mbf{Z}\\ \wt{b_i}\in \mbf{Z}/rr_i^{\ast}\mbf{Z}\\(b_i,\wt{q_i}q_i^{\ast}) =1\\(\wt{b_i},rr_i^{\ast}) = 1}}F\bigg(\sum_{i=1}^{2}\frac{b_i}{q_i} + \frac{\wt{b_i}}{rr_i^{\ast}}\bigg).
\end{align*}
Via \cref{lem:frac}, we may combine the $\sum_{i=1}^{2}b_i/q_i$; note that the conditions apply separately for each prime and thus the above is bounded by
\begin{align*}
W(\log\log N)^{O(1)}\sum_{\substack{r,r_1^{\ast},r_2^{\ast}\le Q\\(rr_i^{\ast},q_i) =1\\\mu(r_1^\ast), \mu(r_2^{\ast}),\mu(r) \neq 0\\ (r,r_i^{\ast}) = (r_1^{\ast},r_2^{\ast}) = 1}}\frac{(q_1,q_2)}{(q_1q_2)^{1/2}(r^2r_1^{\ast}r_2^{\ast})}\sum_{\substack{b\in \mbf{Z}/\wh{q}\wt{q}q_1^{\ast}q_2^{\ast}\mbf{Z}\\ \wt{b_i}\in \mbf{Z}/rr_i^{\ast}\mbf{Z}\\(b,\wt{q}q_1^{\ast}q_2^{\ast}) =1\\(\wt{b_i},rr_i^{\ast}) = 1}}F\bigg(\frac{b}{\wh{q}\wt{q}q_1^{\ast}q_2^{\ast}} + \sum_{i=1}^{2} \frac{\wt{b_i}}{rr_i^{\ast}}\bigg).
\end{align*}
Note that 
\[(q_1,q_2)\le (\wh{q_1}\wh{q_2})^{1/2}\cdot (\wt{q_1}\wt{q_2})^{1/2-1/38}\]
and thus the above is bounded by 
\begin{align*}
W(\log\log N)^{O(1)}\sum_{\substack{r,r_1^{\ast},r_2^{\ast}\le Q\\(rr_i^{\ast},q_i) =1\\\mu(r_1^\ast), \mu(r_2^{\ast}),\mu(r) \neq 0\\ (r,r_i^{\ast}) = (r_1^{\ast},r_2^{\ast}) = 1}}\frac{1}{(\wt{q}_1\wt{q}_2)^{1/38}(q_1^{\ast}q_2^{\ast})^{1/2}(r^2r_1^{\ast}r_2^{\ast})}\sum_{\substack{b\in \mbf{Z}/\wh{q}\wt{q}q_1^{\ast}q_2^{\ast}\mbf{Z}\\ \wt{b_i}\in \mbf{Z}/rr_i^{\ast}\mbf{Z}\\(b,\wt{q}q_1^{\ast}q_2^{\ast}) =1\\(\wt{b_i},rr_i^{\ast}) = 1}}F\bigg(\frac{b}{\wh{q}\wt{q}q_1^{\ast}q_2^{\ast}} + \sum_{i=1}^{2} \frac{\wt{b_i}}{rr_i^{\ast}}\bigg).
\end{align*}
We next reduce the fraction with numerator $b$ by choosing $\wh{q}' = (b,\wh{q})|\wh{q}$l; the above is thus bounded by 
\begin{align*}
W(\log\log N)^{O(1)}\sum_{\substack{r,r_1^{\ast},r_2^{\ast}\le Q\\(rr_i^{\ast},q_i) =1\\\mu(r_1^\ast), \mu(r_2^{\ast}),\mu(r) \neq 0\\ (r,r_i^{\ast}) = (r_1^{\ast},r_2^{\ast}) = 1\\\wh{q}'|\wh{q}}}\frac{1}{(\wt{q}_1\wt{q}_2)^{1/38}(q_1^{\ast}q_2^{\ast})^{1/2}(r^2r_1^{\ast}r_2^{\ast})}\sum_{\substack{b\in \mbf{Z}/\wh{q}'\wt{q}q_1^{\ast}q_2^{\ast}\mbf{Z}\\ \wt{b_i}\in \mbf{Z}/rr_i^{\ast}\mbf{Z}\\(b,\wh{q}'\wt{q}q_1^{\ast}q_2^{\ast}) =1\\(\wt{b_i},rr_i^{\ast}) = 1}}F\bigg(\frac{b}{\wh{q}'\wt{q}q_1^{\ast}q_2^{\ast}} + \sum_{i=1}^{2} \frac{\wt{b_i}}{rr_i^{\ast}}\bigg).
\end{align*}
We next reduce the fractions with denominators $rr_i^{\ast}$. As $(r_1^{\ast},r_2^{\ast}) = (r,r_i^{\ast}) = 1$ we may apply \cref{lem:frac} and obtain that the above is bounded by
\begin{align*}
&W(\log\log N)^{O(1)}\sum_{\substack{r,r_1^{\ast},r_2^{\ast}\le Q\\(rr_i^{\ast},q_i) =1\\\mu(r_1^\ast), \mu(r_2^{\ast}),\mu(r) \neq 0\\ (r,r_i^{\ast}) = (r_1^{\ast},r_2^{\ast}) = 1\\\wh{q}'|\wh{q}}}\frac{r}{(\wt{q}_1\wt{q}_2)^{1/38}(q_1^{\ast}q_2^{\ast})^{1/2}(r^2r_1^{\ast}r_2^{\ast})}\sum_{\substack{b\in \mbf{Z}/\wh{q}'\wt{q}q_1^{\ast}q_2^{\ast}\mbf{Z}\\ \wt{b}\in \mbf{Z}/rr_1^{\ast}r_2^{\ast}\mbf{Z}\\(b,\wh{q}'\wt{q}q_1^{\ast}q_2^{\ast}) =1\\(\wt{b},r_1^{\ast}r_2^{\ast}) = 1}}F\bigg(\frac{b}{\wh{q}'\wt{q}q_1^{\ast}q_2^{\ast}} + \frac{\wt{b}}{rr_1^{\ast}r_2^{\ast}}\bigg)\\
&= W(\log\log N)^{O(1)}\sum_{\substack{r,r_1^{\ast},r_2^{\ast}\le Q\\(rr_i^{\ast},q_i) =1\\\mu(r_1^\ast), \mu(r_2^{\ast}),\mu(r) \neq 0\\ (r,r_i^{\ast}) = (r_1^{\ast},r_2^{\ast}) = 1\\\wh{q}'|\wh{q}}}\frac{1}{(\wt{q}_1\wt{q}_2)^{1/38}(q_1^{\ast}q_2^{\ast})^{1/2}(rr_1^{\ast}r_2^{\ast})}\sum_{\substack{b\in \mbf{Z}/\wh{q}'\wt{q}q_1^{\ast}q_2^{\ast}\mbf{Z}\\ \wt{b}\in \mbf{Z}/rr_1^{\ast}r_2^{\ast}\mbf{Z}\\(b,\wh{q}'\wt{q}q_1^{\ast}q_2^{\ast}) =1\\(\wt{b},r_1^{\ast}r_2^{\ast}) = 1}}F\bigg(\frac{b}{\wh{q}'\wt{q}q_1^{\ast}q_2^{\ast}} + \frac{\wt{b}}{rr_1^{\ast}r_2^{\ast}}\bigg).
\end{align*}
We now choose $r' = (\wt{b},r)|r$ and note that 
\[\sum_{\substack{r'|r\\r\le Q}}\frac{1}{r} \ll \frac{\log N}{r'}.\]
Therefore, the above is bounded by 
\begin{align*}
W(\log N)^{1+o(1)}\sum_{\substack{r,r_1^{\ast},r_2^{\ast}\le Q\\(rr_i^{\ast},q_i) =1\\\mu(r_1^\ast), \mu(r_2^{\ast}),\mu(r) \neq 0\\ (r,r_i^{\ast}) = (r_1^{\ast},r_2^{\ast}) = 1\\\wh{q}'|\wh{q}}}\frac{1}{(\wt{q}_1\wt{q}_2)^{1/38}(q_1^{\ast}q_2^{\ast})^{1/2}(rr_1^{\ast}r_2^{\ast})}\sum_{\substack{b\in \mbf{Z}/\wh{q}'\wt{q}q_1^{\ast}q_2^{\ast}\mbf{Z}\\ \wt{b}\in \mbf{Z}/rr_1^{\ast}r_2^{\ast}\mbf{Z}\\(b,\wh{q}'\wt{q}q_1^{\ast}q_2^{\ast}) =1\\(\wt{b},rr_1^{\ast}r_2^{\ast}) = 1}}F\bigg(\frac{b}{\wh{q}'\wt{q}q_1^{\ast}q_2^{\ast}} + \frac{\wt{b}}{rr_1^{\ast}r_2^{\ast}}\bigg).
\end{align*}
We now let $(r_i^{\ast},q_{3-i}^{\ast}) = q_i'|q_{3 - i}^{\ast}$ and write $r_i^{\ast} = q_i' r_i'$. As there are at most $\tau(q_1^{\ast}q_2^{\ast}) = (q_1^{\ast}q_2^{\ast})^{o(1)}$ such possibilities, the above is bounded by 
\begin{align*}
\sup_{q_i'|q_i^{\ast}}W(\log N)^{1+o(1)}\sum_{\substack{r,r_1',r_2'\le Q\\(rr_1'r_2',q_1q_2) =1\\\mu(r) , \mu(r_i') \neq 0\\ (r,r_i') = (r_1',r_2') = 1\\\wh{q}'|\wh{q}}}\frac{1}{(\wt{q}_1\wt{q}_2)^{1/38}(q_1^{\ast}q_2^{\ast})^{1/2-o(1)}(rq
_1'q_2'r_1'r_2')}\sum_{\substack{b\in \mbf{Z}/\wh{q}'\wt{q}q_1^{\ast}q_2^{\ast}\mbf{Z}\\ \wt{b}\in \mbf{Z}/rq_1'q_2'r_1'r_2'\mbf{Z}\\(b,\wh{q}'\wt{q}q_1^{\ast}q_2^{\ast}) =1\\(\wt{b},rr_1^{\ast}r_2^{\ast}) = 1}}F\bigg(\frac{b}{\wh{q}'\wt{q}q_1^{\ast}q_2^{\ast}} + \frac{\wt{b}}{rq_1'q_2'r_1'r_2'}\bigg).
\end{align*}
We now apply \cref{lem:frac} to obtain that the above is bounded by 
\begin{align*}
&\sup_{q_i'|q_i^{\ast}}W(\log N)^{1+o(1)}\sum_{\substack{r,r_1',r_2'\le Q\\(rr_1'r_2',q_1q_2) =1\\\mu(r) , \mu(r_i') \neq 0\\ (r,r_i') = (r_1',r_2') = 1\\\wh{q}'|\wh{q}}}\frac{1}{(\wt{q}_1\wt{q}_2)^{1/38}(q_1^{\ast}q_2^{\ast})^{1/2-o(1)}(rr
_1'r_2')}\sum_{\substack{b\in \mbf{Z}/\wh{q}'\wt{q}q_1^{\ast}q_2^{\ast}rr_1'r_2'\mbf{Z}\\(b,\wh{q}'\wt{q}rr_1'r_2') =1}}F\bigg(\frac{b}{\wh{q}'\wt{q}q_1^{\ast}q_2^{\ast}rr_1'r_2'}\bigg)\\
&\ll W(\log N)^{1+o(1)}\sum_{\substack{r,r_1',r_2'\le Q\\(rr_1'r_2',q_1q_2) =1\\\mu(r) , \mu(r_i') \neq 0\\ (r,r_i') = (r_1',r_2') = 1\\\wh{q}'|\wh{q}}}\frac{1}{(\wt{q}_1\wt{q}_2)^{1/38}(q_1^{\ast}q_2^{\ast})^{1/2-o(1)}(rr
_1'r_2')}\sum_{\substack{b\in \mbf{Z}/\wh{q}'\wt{q}q_1^{\ast}q_2^{\ast}rr_1'r_2'\mbf{Z}\\(b,\wh{q}'\wt{q}rr_1'r_2') =1}}F\bigg(\frac{b}{\wh{q}'\wt{q}q_1^{\ast}q_2^{\ast}rr_1'r_2'}\bigg)\\
&\ll \sup_{q'|q_1^{\ast}q_2^{\ast}} W(\log N)^{1+o(1)}\sum_{\substack{r,r_1',r_2'\le Q\\(rr_1'r_2',q_1q_2) =1\\\mu(r) , \mu(r_i') \neq 0\\ (r,r_i') = (r_1',r_2') = 1\\\wh{q}'|\wh{q}}}\frac{1}{(\wt{q}_1\wt{q}_2)^{1/38}(q_1^{\ast}q_2^{\ast})^{1/2-o(1)}(rr
_1'r_2')}\sum_{\substack{b\in \mbf{Z}/\wh{q}'\wt{q}q'rr_1'r_2'\mbf{Z}\\(b,\wh{q}'\wt{q}q'rr_1'r_2') =1}}F\bigg(\frac{b}{\wh{q}'\wt{q}q'rr_1'r_2'}\bigg).
\end{align*}
Using the divisor bound and that $(r,r_i') = (r_1',r_2') = 1$, we may combine $r$, $r_1'$, and $r_2'$ into a single variable and bound the above by 
\begin{equation}\label{eq:reduc-final}
\sup_{q'|q_1^{\ast}q_2^{\ast}} W(\log N)^{1+o(1)}\sum_{\substack{r\le Q^3\\(r,q_1q_2) =1\\\mu(r) \neq 0\\\wh{q}'|\wh{q}}}\frac{1}{(\wt{q}_1\wt{q}_2)^{1/38}(q_1^{\ast}q_2^{\ast})^{1/2-o(1)}r^{1-o(1)}}\sum_{\substack{b\in \mbf{Z}/\wh{q}'\wt{q}q'r\mbf{Z}\\(b,\wh{q}'\wt{q}q'r) =1}}F\bigg(\frac{b}{\wh{q}'\wt{q}q'r}\bigg).
\end{equation}

\textbf{Step 4: Applying large--sieve estimates}

We are now in position to invoke \cref{lem:well-output}. Recall that 
\[W\le \frac{\exp(O((\log\log N)^{5/2}))}{\wh{q}^{1/9}}.\]
Therefore, the sum which we seek to bound is bounded by 
\begin{align*}
\sup_{q'|q_1^{\ast}q_2^{\ast}} \exp(O((\log\log N)^{5/2}))\sum_{\substack{r\le Q^3\\(r,q_1q_2) =1\\\mu(r) \neq 0\\\wh{q}'|\wh{q}}}\frac{1}{\wh{q}^{1/9}(\wt{q}_1\wt{q}_2)^{1/38}(q_1^{\ast}q_2^{\ast})^{1/2-o(1)}r^{1-o(1)}}\sum_{\substack{b\in \mbf{Z}/\wh{q}'\wt{q}q'r\mbf{Z}\\(b,\wh{q}'\wt{q}q'r) =1}}F\bigg(\frac{b}{\wh{q}'\wt{q}q'r}\bigg).
\end{align*}
Let $r\sim R$ and $\hat{q}'\sim L$; we have 
\begin{align*}
&\sup_{q'|q_1^{\ast}q_2^{\ast}}\sum_{\substack{r\sim R\\(r,q_1q_2) =1\\\mu(r) \neq 0\\\wh{q}'|\wh{q},\wh{q}'\sim L}}\frac{\exp(O((\log\log N)^{5/2}))}{\wh{q}^{1/9}(\wt{q}_1\wt{q}_2)^{1/38}(q_1^{\ast}q_2^{\ast})^{1/2-o(1)}r^{1-o(1)}}\sum_{\substack{b\in \mbf{Z}/\wh{q}'\wt{q}q'r\mbf{Z}\\(b,\wh{q}'\wt{q}q'r) =1}}F\bigg(\frac{b}{\wh{q}'\wt{q}q'r}\bigg)\\
&\ll \sup_{q'|q_1^{\ast}q_2^{\ast}}\frac{\exp(O((\log\log N)^{5/2}))}{(\wt{q}_1\wt{q}_2)^{1/38}(q_1^{\ast}q_2^{\ast})^{1/2-o(1)}\hat{q}^{1/18}L^{1/18-o(1)}R^{1-o(1)}}\sum_{\substack{b\in \mbf{Z}/(\wt{q}q'T\mbf{Z})\\(b,\wt{q}q'T) = 1\\LR/4 \le T \le 4LR}}F\bigg(\frac{b}{\wt{q}q'T}\bigg)\\
&\ll \sup_{q'|q_1^{\ast}q_2^{\ast}}\frac{\exp(O((\log\log N)^{5/2}))}{(\wt{q}_1\wt{q}_2)^{1/38}(q_1^{\ast}q_2^{\ast})^{1/2-o(1)}\hat{q}^{1/18}L^{1/18-o(1)}R^{1-o(1)}}\big((LR)^2 \cdot \wt{q}q'\big)^{1/50}
\end{align*}
where we have applied the large sieve portion of \cref{lem:well-output}. As $\wt{q}\le \wt{q}_1\wt{q}_2$ and $q'\le (q_1^{\ast}q_2^{\ast})$, the above is bounded by (say)
\[\frac{\exp((\log\log N)^{O(1)})}{(q_1q_2)^{1/200}}.\]
The lemma follows unless both of the moduli are less than $\exp((\log\log N)^{O(1)})$. By \cref{lem:large-conduc} this forces either $q_1 = q_2$ or $q_i = 1$ for one of $i\in \{1,2\}$.

\textbf{Step 5: Handling degenerate cases}

We first handle the case where $q_1 = 1$ and $q_2\neq 1$. In this case we have that $\hat{q_1} = \hat{q_2} = \wt{q}_1 = \wt{q}_2 = q_1^{\ast} = 1$ and $q_2^{\ast} = q_2$. Thus, returning to \cref{eq:reduc-final}, we are seeking to bound
\begin{align*}
\sup_{q'|q_2} W(\log N)^{1+o(1)}\sum_{\substack{r\le Q^3\\(r,q_2) =1\\\mu(r) \neq 0}}\frac{1}{(q_2)^{1/2-o(1)}r^{1-o(1)}}\sum_{\substack{b\in \mbf{Z}/q'r\mbf{Z}\\(b,q'r) =1}}F\bigg(\frac{b}{q'r}\bigg).
\end{align*}
Noting that $W\le 1$, taking $r\sim R$, and applying \cref{lem:well-output},
\begin{align*}
&\sup_{q'|q_2} (\log N)^{1+o(1)}\sum_{\substack{r\sim R\\(r,q_2) =1\\\mu(r) \neq 0}}\frac{1}{(q_2)^{1/2-o(1)}r^{1-o(1)}}\sum_{\substack{b\in \mbf{Z}/q'r\mbf{Z}\\(b,q'r) =1}}F\bigg(\frac{b}{q'r}\bigg)\\
&\ll \sup_{q'|q_2} (\log N)^{1+o(1)}\frac{1}{(q_2)^{1/2-o(1)}R^{1-o(1)}}\cdot (R^2q')^{1/50}\ll (\log N)^{1+o(1)}q_2^{-1/4 + o(1)}R^{-1/2}.
\end{align*}
Summing over dyadic pieces and using that $q_2\gg_{B} (\log N)^{B}$ by \cref{lem:large-conduc} concludes this exceptional case. 

Finally, we handle the case $q_1 = q_2 = q$ and we may assume that $q\le \exp(O((\log\log N)^{5/2}))$ as otherwise the argument from the previous step suffices. In this case, $\wh{q_1} = \wh{q_2} = q_i$ and $\wt{q}_1 = \wt{q}_2 = q_1^{\ast} = q_2^{\ast} = 1$. Thus, our bound becomes 
\begin{align*}
 W(\log N)^{1+o(1)}\sum_{\substack{r\le Q^3\\(r,q_1) =1\\\mu(r) \neq 0\\\wh{q}'|q_1}}\frac{1}{r^{1-o(1)}}\sum_{\substack{b\in \mbf{Z}/\wh{q}'r\mbf{Z}\\(b,\wh{q}'r) =1}}F\bigg(\frac{b}{\wh{q}'r}\bigg).
\end{align*}
Using that $W\le 1$, $q' \le \exp(O(\log\log N))^{5/2})$, and \cref{lem:well-output}, the contribution is negligible unless $r\le \exp(O((\log\log N)^{5/2}))$; therefore it suffices to bound 
\begin{align*}
 W(\log N)^{1+o(1)}\sum_{\substack{r\le \exp(O((\log\log N)^{5/2}))\\(r,q_1) =1\\\mu(r) \neq 0\\\wh{q}'|q_1}}\frac{1}{r^{1-o(1)}}\sum_{\substack{b\in \mbf{Z}/\wh{q}'r\mbf{Z}\\(b,\wh{q}'r) =1}}F\bigg(\frac{b}{\wh{q}'r}\bigg).
\end{align*}
In this case note that $\wh{q}'r\le \exp(O((\log\log N)^{5/2}))$. Note via the $L^\infty$ estimate of \cref{lem:well-output}, we have that $F\big(\frac{b}{\wh{q}'r}\big)\le \exp(-(\log N)^{1/3})$ unless all prime factors of $r$ and $\wh{q}'$ are prime factors of $g$. Furthermore as $\mu(r) \neq 0$, we in fact have that $r|g$ and it suffices to bound
\begin{align*}
 W(\log N)^{1+o(1)}\sum_{\substack{\wh{q}'|g\prod_{p|g}p^{v_p(q)}\\\wh{q}'\le \exp(O((\log\log N)^{5/2}))}}\sum_{\substack{b\in \mbf{Z}/\wh{q}'\mbf{Z}\\(b,\wh{q}') =1}}F\bigg(\frac{b}{\wh{q}'}\bigg).
\end{align*}
We now break into cases. If 
\[\prod_{p|g}p^{v_p(q)}\ge (\log N)^{16gC'},\]
then note that 
\[W\le \frac{(\log N)^{gC'}}{\prod_{p|g}p^{v_p(q)/8}}\le \frac{1}{\prod_{p|g}p^{v_p(q)/16}}.\]
As the sum we are bounding is controlled by 
\begin{align*}
W(\log N)^{1+o(1)}\sum_{\substack{\wh{q}'|g\prod_{p|g}p^{v_p(q)}\\\wh{q}'\le \exp(O((\log\log N)^{5/2}))}}\sum_{\substack{b\in \mbf{Z}/\wh{q}'\mbf{Z}\\(b,\wh{q}') =1}}F\bigg(\frac{b}{\wh{q}'}\bigg),
\end{align*}
the result follows via applying the large sieve bound in \cref{lem:well-output}. Finally, if 
\[\prod_{p|g}p^{v_p(q)}\le (\log N)^{16gC'},\] then note that desired sum is bounded by 
\[W(\log N)^{128gC'}\]
and that
\[W\le \bigg(\frac{q}{(T,q)}\bigg)^{-1/32}\le (\log N)^{-C'^2}\]
as desired.
\end{proof}
We are now finally in position to prove \cref{lem:correction}. 
\begin{proof}[Proof of \cref{lem:correction}.]
First fix $(\rho_1, \chi_{1})$ and $(\rho_2, \chi_{2})$ and note that we may assume by symmetry that $\chi_{2}$ is nontrivial. Furthermore, let $q_1$ and $q_2$ be the corresponding moduli. 

Note that $|x_j^{\rho_j-1}F_{\chi_{j},Q}(x_j)|\ll \tau(x_j) (\log M)^3$. Let $G(x_1,x_2,x_3)$ be any function; note that 
\begin{align*}
\sum_{x_3\in \mc{S}_b}\tau(x_3)&\bigg|\sum_{\substack{x_1+x_2 = N-x_3\\x_i\in \mc{S}_b}}\prod_{j=1}^{2}x_j^{\rho_j-1}F_{\chi_j,Q}(x_j)G(x_1,x_2,x_3)\bigg|\\
&\ll (\log M)^{O(1)} \sum_{\substack{x + x_2 + x_3 = N\\x_i\in \mc{S}_b}}|G(x_1,x_2,x_3)|\prod_{i=1}^{3}\tau(x_i)\\
&\ll (\log M)^{O(1)} \bigg(\sum_{\substack{x + x_2 + x_3 = N\\x_i\in \mc{S}_b}}|G(x_1,x_2,x_3)|^2\bigg)^{1/2}\bigg(\sum_{\substack{x + x_2 + x_3 = N\\x_i\in \mc{S}_b}}\prod_{i=1}^{3}\tau(x_i)^2\bigg)^{1/2}\\
&\ll (\log M)^{O(1)} \bigg(\sum_{\substack{x + x_2 + x_3 = N\\x_i\in \mc{S}_b}} 1\bigg)^{1/2}\cdot \bigg(\sum_{\substack{x + x_2 + x_3 = N\\x_i\in \mc{S}_b}}|G(x_1,x_2,x_3)|^2\bigg)^{1/2}.
\end{align*}
Therefore if $G(x_1,x_2,x_3)$ is a $\{0,1\}$ function (e.g. the indicator of an event) and it occurs with sufficient small logarithmic probability, we may exclude it from the analysis. 

We must now decompose $(x_1| x_1 + x_2 + x_3 = N)$ into measures $\mu_\mc{B}$ that have a product-like structure which we will later show is well-conditioned. Conditioning on the carries, we reveal $(x_{1,j}+x_{2,j},x_{3,j})$ for $j\in [k/2]$ and $(x_{1,j},x_{2,j},x_{3,j})$ for $j>k/2$. (We have adopted the digit conventions in \cref{lem:chernoff}.) Note that given this information that $\mc{B}_j$ for $j\in [k/2]$ is an interval (missing at most $2$ elements) of length $\ge \min(x_{1,j}+x_{2,j},2g-x_{1,j}-x_{2,j}) - O(1)$. Furthermore, given this information we may determine $x_1 + x_2$. Therefore in order to complete the proof it suffices to guarantee the assumption of \cref{lem:key-computation} holds with sufficiently high probability. 

Let $C$ be a sufficiently large constant in terms of $A$ and $B$ to be chosen later. Taking $G(x_1,x_2,x_3) = \mbm{1}\big[\tau(N-x_3)\ge (\log M)^{C}\big]$ and applying \cref{lem:divisor-function} proves that we may restrict attention to $\tau(x_1+x_2) = \tau(N-x_3)\le (\log M)^{C}$. Furthermore, by \cref{lem:divis}, we may assume that for all $p\le (\log M)^{C}$ that $v_p(T)\le C(\log\log M)$ and therefore,
\[\prod_{\substack{p|T\\p\le (\log N)^C}}p^{v_p(T)}\le \exp((\log\log N)^{5/2}).\]
Analogously, via \cref{lem:divis} (and that $q_1q_2$ has at most $O(\log M)$ distinct prime factors), we may assume that if $p|(x_1+x_2,q_1q_2)$, then $p\le (\log M)^{C}$. Furthermore, by \cref{lem:divis-small} and \cref{lem:large-conduc}, we may assume that $q_2/(T,q_2)\ge (\log N)^{C}$.

These conditions handle the necessary divisibility constraints on $T$ in \cref{lem:key-computation}. It remains to guarantee that $\mc{B}$ is well-conditioned. That the $\mc{B}_i$ are intervals missing at most $2$ points will always follow by construction. To prove that $\mc{B}$ satisfies the first item of \cref{def:well-conditioned}, we examine when $\min(x_{1,j}+x_{2,j},2g-x_{1,j}-x_{2,j})\le 2g^{99/100}$ for $j\in [k/2]$. Via the first item of \cref{lem:chernoff} and \cref{lem:chernoff-hoeff}, we see that the probability that the first item fails is proportional to $(\log N)^{-O(C)}$ which gives the desired result. To prove that the uniform measure on $\mc{B}$ satisfies the third item of \cref{def:well-conditioned}, note that 
\[\sum_{j\in [k/8,k/4]}\mbm{1}[\snorm{a/b\cdot g^{j}}\ge 1/g^2]\ge \frac{k}{20\log b}.\]
To prove this, note that $\snorm{a/b\cdot g^{j}}\ge 1/b$ for all $j$ and thus in any interval of length $\log b/\log g$ there is at least $1$ such index. Union-bounding over all choices of $1\le a<b$ with $(b,g) = 1$ along with first item of \cref{lem:chernoff} and \cref{lem:chernoff-hoeff} implies that this event fails with probability at most (say) $\exp(-(\log M)^{1/2})$ which is certainly acceptable. Finally, note that $x_{1,j},x_{2,j}\notin \{g-1,0,1\}$ for at least $1$ digit $j\in [2k/3,3k/4]$ with probability $N^{-\Omega(1)}$; this certifies the last condition in \cref{lem:key-computation}.

Therefore up to tolerable losses, we may assume we are in the situation of \cref{lem:key-computation} and the result follows. 
\end{proof}

\subsection{Computing main term}
We now give the proof of \cref{lem:main-term}. This is essentially a simplification of \cref{lem:correction} and in particular is able to avoid character sum estimates used in the previous part. 

We first prove a simpler variant of \cref{lem:key-computation} which will be used in the proof of \cref{lem:main-term}.
\begin{lemma}\label{lem:key-computation2}
Fix $B\ge 1$; there exists $C = C_{\ref{lem:key-computation2}}\ge 1$ and $C' = C_{\ref{lem:key-computation2}}'(B,g)\ge g^2B$ such that the following holds. Let $g\ge C$ and $x_1\sim \mu_{\vec{\mc{B}}}$ with $\vec{\mc{B}}$ where $\mc{B}$ is $C'$ well-conditioned and $|\mc{B}_j| = 1$ for $j\in [k/2,k]$.

Then 
\begin{align*}
\bigg|\sum_{\substack{x_1 + x_2 = T\\x_1\sim \mu_{\vec{\mc{B}}}}}\prod_{j=1}^{2}\Lambda_{Q}(x_i) - &\sum_{\substack{r_1'|g,r_2'|g\\r_1'r,r_2'r\le Q\\(g,r) =1}}\frac{\mu(r)^2\prod_{i=1}^{2}\mu(r_i')}{\phi(r)^2\prod_{i=1}^{2}\phi(r_i')}\sum_{\substack{b\in (\mbf{Z}/r\mbf{Z})^{\ast}\\b_i'\in (\mbf{Z}/r_i'\mbf{Z})^{\ast}}}\sum_{\substack{x_1 + x_2 = T\\x_1\sim \mu_{\vec{\mc{B}}}}}e\bigg(\frac{bT}{r} + \sum_{i=1}^{2}\frac{b_i'x_i}{r_i'}\bigg)\bigg|\\
&\ll_{A,B,g}(\log M)^{-B}\cdot \prod_{i\in [k]}|\mc{B}_i|
\end{align*}
and 
\begin{align*}
\sup_{x_1 + x_2 = T}\bigg|\sum_{\substack{r_1'|g,r_2'|g\\r_1'r,r_2'r\le Q\\(g,r) =1}}\frac{\mu(r)^2\prod_{i=1}^{2}\mu(r_i')}{\phi(r)^2\prod_{i=1}^{2}\phi(r_i')}\sum_{\substack{b\in (\mbf{Z}/r\mbf{Z})^{\ast}\\b_i'\in (\mbf{Z}/r_i'\mbf{Z})^{\ast}}}e\bigg(\frac{bT}{r} + \sum_{i=1}^{2}\frac{b_i'x_i}{r_i'}\bigg)\bigg|\ll_{g}(\log M)^2.
\end{align*}
\end{lemma}
\begin{proof}
Recalling the definition of $\Lambda_Q$ from \eqref{eq:ApproximantMain}, we have that
\begin{align*}
&\sum_{\substack{x_1 + x_2 = T\\x_1\sim \mu_{\vec{\mc{B}}}}}\prod_{j=1}^{2}\Lambda_{Q}(x_j) = \sum_{\substack{x_1 + x_2 = T\\x_1\sim \mu_{\vec{\mc{B}}}}}\sum_{r_1,r_2\le Q}\frac{\prod_{i=1}^{2}\mu(r_i)}{\prod_{i=1}^{2}\phi(r_i)}\sum_{\substack{b_i\in (\mbf{Z}/r_i\mbf{Z})^{\ast}}}e\bigg(\sum_{i=1}^{2}\frac{b_ix_i}{r_i}\bigg).
\end{align*}
We now define $\mc{S}$ to be the set of fraction $\frac{a}{b}$ with $(a,b) = 1$ and $b|g$. We first bound the contribution from all terms with $b_1/r_1-b_2/r_2\notin \mc{S}$. Since the M\"obius function is supported on squarefree numbers, the contribution is bounded by
\begin{align*}
&\sum_{\substack{r_1,r_2\le Q\\r_1'|g,r_2'|g\\ (r_1,r_1') = (r_2,r_2') = (r_1r_2, g) = 1}}\frac{\prod_{i=1}^{2}|\mu(r_i)|\cdot |\mu(r_i')|}{\prod_{i=1}^{2}\phi(r_i)\phi(r_i')}\sum_{\substack{b_i\in (\mbf{Z}/r_i\mbf{Z})^{\ast}\\ b_i'\in (\mbf{Z}/r_i'\mbf{Z})^{\ast}\\b_1/r_1-b_2/r_2\notin \mc{S}}}\bigg|\sum_{\substack{x_1 + x_2 = T\\x_1\sim \mu_{\vec{\mc{B}}}}}e\bigg(\sum_{i=1}^{2}\frac{b_ix_i}{r_i} + \frac{b_i'x_i}{r_i'}\bigg)\bigg|\\
&\ll_{g} (\log\log N)^{O(1)}\sup_{\beta\in \mbf{R}/\mbf{Z}}\sum_{\substack{r_1,r_2\le Q\\ (r_1r_2, g) = 1}}\frac{|\mu(r_1)|\cdot |\mu(r_2)|}{r_1r_2}\sum_{\substack{b_i\in (\mbf{Z}/r_i\mbf{Z})^{\ast}\\b_1/r_1+b_2/r_2\notin \mc{S}}}\bigg|F\bigg(\sum_{i=1}^{2}\frac{b_i}{r_i} + \beta\bigg)\bigg|.
\end{align*}
We may write $r_1 = r r_1^{\ast}$ and $r_2 = r r_2^{\ast}$ with $(r,r_1^{\ast}) = (r,r_2^{\ast}) = (r_1^{\ast},r_2^{\ast}) = 1$. With this notation, the above is bounded by
\begin{align*}
(\log\log N)^{O(1)}\sup_{\beta\in \mbf{R}/\mbf{Z}}\sum_{\substack{r,r_1^{\ast},r_2^{\ast}\le Q\\ (rr_1^{\ast}r_2^{\ast}, g) = (r,r_1^{\ast}) = (r,r_2^{\ast}) = (r_1^{\ast},r_2^{\ast}) = 1}}\frac{1}{r^2r_1^{\ast}r_2^{\ast}}\sum_{\substack{b_i\in (\mbf{Z}/rr_i^{\ast}\mbf{Z})^{\ast}\\b_1/(rr_1^{\ast})+b_2/(rr_2^{\ast})\notin \mc{S}}}\bigg|F\bigg(\sum_{i=1}^{2}\frac{b_i}{rr_i^{\ast}} + \beta\bigg)\bigg|.
\end{align*}
Combining denominators (via \cref{lem:frac}) and reducing the fraction with respect to $r$, we obtain that the above is bounded by 
\begin{align*}
(\log\log N)^{1+O(1)}\sup_{\beta\in \mbf{R}/\mbf{Z}}\sum_{\substack{r,r_1^{\ast},r_2^{\ast}\le Q\\ (rr_1^{\ast}r_2^{\ast}, g) = (r,r_1^{\ast}) = (r,r_2^{\ast}) = (r_1^{\ast},r_2^{\ast}) = 1}}\frac{1}{rr_1^{\ast}r_2^{\ast}}\sum_{\substack{b\in (\mbf{Z}/rr_1^{\ast}r_2^{\ast}\mbf{Z})^{\ast}\\b/rr_1^{\ast}r_2^{\ast}\notin \mc{S}}}\bigg|F\bigg(\frac{b}{rr_1^{\ast}r_2^{\ast}} + \beta\bigg)\bigg|.
\end{align*}
Combining $rr_1^{\ast}r_2^{\ast}$, we obtain that the above is bounded by 
\begin{align*}
(\log\log N)^{1+O(1)}\sup_{\beta\in \mbf{R}/\mbf{Z}}\sum_{\substack{r\le Q^{3}\\ (r, g) = 1}}\frac{1}{r^{1-o(1)}}\sum_{\substack{b\in (\mbf{Z}/r\mbf{Z})^{\ast}\\b/r\notin \mc{S}}}\bigg|F\bigg(\frac{b}{r} + \beta\bigg)\bigg|.
\end{align*}
If $r\ge (\log N)^{O(C)}$, then the large sieve inequality in \cref{lem:well-output} completes the proof. Furthermore, noting that $\beta$ in the above analysis may be taken to be a fraction with denominator dividing $g$, we immediately have the desired result by the $L^{\infty}$-estimate in \cref{lem:well-output} unless $r = 1$. This forces $b_1/r_1 - b_2/r_2\in \mc{S}$ and thus we may restrict attention to this case. 

Therefore, it suffices to analyze
\begin{align*}
&\sum_{r_1,r_2\le Q}\frac{\prod_{i=1}^{2}\mu(r_i)}{\prod_{i=1}^{2}\phi(r_i)}\sum_{\substack{b_i\in (\mbf{Z}/r_i\mbf{Z})^{\ast}\\b_1/r_1-b_2/r_2\in \mc{S}}}\sum_{\substack{x_1 + x_2 = T\\x_1\sim \mu_{\vec{\mc{B}}}}}e\bigg(\sum_{i=1}^{2}\frac{b_ix_i}{r_i}\bigg)\\
&= \sum_{\substack{r_1'|g,r_2'|g\\r_1'r_1,r_2'r_2\le Q\\(g,r_1) = (g,r_2) = 1}}\frac{\prod_{i=1}^{2}\mu(r_i)\mu(r_i')}{\prod_{i=1}^{2}\phi(r_i)\phi(r_i')}\sum_{\substack{b_i\in (\mbf{Z}/r_i\mbf{Z})^{\ast}\\b_i'\in (\mbf{Z}/r_i'\mbf{Z})^{\ast}\\b_1/r_1-b_2/r_2=0}}\sum_{\substack{x_1 + x_2 = T\\x_1\sim \mu_{\vec{\mc{B}}}}}e\bigg(\sum_{i=1}^{2}\frac{b_ix_i}{r_i} + \frac{b_i'x_i}{r_i'}\bigg)\\
&=\sum_{\substack{r_1'|g,r_2'|g\\r_1'r,r_2'r\le Q\\(g,r) =1}}\frac{\mu(r)^2\prod_{i=1}^{2}\mu(r_i')}{\phi(r)^2\prod_{i=1}^{2}\phi(r_i')}\sum_{\substack{b\in (\mbf{Z}/r\mbf{Z})^{\ast}\\b_i'\in (\mbf{Z}/r_i'\mbf{Z})^{\ast}}}\sum_{\substack{x_1 + x_2 = T\\x_1\sim \mu_{\vec{\mc{B}}}}}e\bigg(\frac{bT}{r} + \sum_{i=1}^{2}\frac{b_i'x_i}{r_i'}\bigg)
\end{align*}
as desired. (Note that the first equality follows since $r_i$ are relatively prime to $g$ so we must have $b_1/r_1 - b_2/r_2 = 0$.) The second claim follows as 
\begin{align*}
\sup_{x_1 + x_2 = T}&\bigg|\sum_{\substack{r_1'|g,r_2'|g\\r_1'r,r_2'r\le Q\\(g,r) =1}}\frac{\mu(r)^2\prod_{i=1}^{2}\mu(r_i')}{\phi(r)^2\prod_{i=1}^{2}\phi(r_i')}\sum_{\substack{b\in (\mbf{Z}/r\mbf{Z})^{\ast}\\b_i'\in (\mbf{Z}/r_i'\mbf{Z})^{\ast}}}e\bigg(\frac{bT}{r} + \sum_{i=1}^{2}\frac{b_i'x_i}{r_i'}\bigg)\bigg|\\
&\le (\log\log M)^{O(1)}\sum_{\substack{r_1'|g,r_2'|g\\r_1'r,r_2'r\le Q\\(g,r) =1}}\frac{1}{r^2}\sum_{\substack{b\in (\mbf{Z}/r\mbf{Z})^{\ast}\\b_i'\in (\mbf{Z}/r_i'\mbf{Z})^{\ast}}} 1\ll_{g}(\log M)^2
\end{align*}
as desired. 
\end{proof}

We now proceed with this proof of \cref{lem:main-term}; the argument is again a simplification of \cref{lem:correction}.
\begin{proof}[{Proof of \cref{lem:main-term}}]
Via decomposing in $C$-well conditioned measures as in \cref{lem:correction}, and noting that the non-well-conditioned cases contribute minimally, it suffices to estimate 
\[\sum_{\substack{x_1+x_2+x_3 = N\\x_i\in \mc{S}_b}}\Lambda_{Q}(x_3)\sum_{\substack{r_1'|g,r_2'|g\\r_1'r,r_2'r\le Q\\(g,r) =1}}\frac{\mu(r)^2\prod_{i=1}^{2}\mu(r_i')}{\phi(r)^2\prod_{i=1}^{2}\phi(r_i')}\sum_{\substack{b\in (\mbf{Z}/r\mbf{Z})^{\ast}\\b_i'\in (\mbf{Z}/r_i'\mbf{Z})^{\ast}}}e\bigg(\frac{b(N-x_3)}{r} + \sum_{i=1}^{2}\frac{b_i'x_i}{r_i'}\bigg).\]
Via expanding out $\Lambda_{Q}(x_3)$, this is equivalent to
\begin{align*}
\sum_{\substack{r_1'|g,r_2'|g,r_3'|g\\r_1'r,r_2'r,r_3'r_2\le Q\\(g,r) = (g,r_3)=1}}\frac{\mu(r)^2\mu(r_3)\prod_{i=1}^{3}\mu(r_i')}{\phi(r)^2\phi(r_3)\prod_{i=1}^{3}\phi(r_i')}\sum_{\substack{b\in (\mbf{Z}/r\mbf{Z})^{\ast}\\b_3\in (\mbf{Z}/r_3\mbf{Z})^{\ast}\\b_i'\in (\mbf{Z}/r_i'\mbf{Z})^{\ast}}}\sum_{\substack{x_1 + x_2 + x_3= N\\x_i\in \mc{S}_b }}e\bigg(\frac{b_3x_3}{r_3}-\frac{bx_3}{r} + \frac{bN}{r} + \sum_{i=1}^{3}\frac{b_i'x_i}{r_i'}\bigg).
\end{align*}
Note that $\sum_{i=1}^{3}b_i'x_i/r_i'$ depends only on the final digit of $(x_1,x_2,x_3)$. This expression is essentially of the form considered in \cref{lem:key-computation2} (where $\mu(r)/\phi(r)$ is replaced by the smaller in magnitude $\mu(r)^2/\phi(r)^2$). Via revealing the last digits of $(x_1,x_2,x_3)$, decomposing into well conditioned measures, and applying an essentially identical (and simpler) analysis, we may restrict attention to the cases $b_3/r_3 = b/r$ (up to an acceptable error). 

Therefore, up to an acceptable error, it suffices to analyze 
\begin{align*}
\sum_{\substack{r_1'|g,r_2'|g,r_3'|g\\r_1'r,r_2'r,r_3'r\le Q\\(g,r) =1}}\frac{\mu(r)^3\prod_{i=1}^{3}\mu(r_i')}{\phi(r)^3\prod_{i=1}^{3}\phi(r_i')}\sum_{\substack{b\in (\mbf{Z}/r\mbf{Z})^{\ast}\\b_i'\in (\mbf{Z}/r_i'\mbf{Z})^{\ast}}}\sum_{\substack{x_1 + x_2 + x_3= N\\x_i\in \mc{S}_b }}e\bigg( \frac{bN}{r} + \sum_{i=1}^{3}\frac{b_i'x_i}{r_i'}\bigg).
\end{align*}
Note as we are summing only $O_g(r)$ Fourier coefficients in the inner sum, the contribution of the terms $r > (\log N)^A$ is at most an $O((\log N)^{-A})$ proportion of the sum. Thus, it suffices to instead analyze the expression 
\begin{align*}
\bigg(\sum_{\substack{r\le (\log N)^{A}\\(g,r) =1}}\frac{\mu(r)^3}{\phi(r)^3}\sum_{b\in (\mbf{Z}/r\mbf{Z})^{\ast}}e\bigg(\frac{bN}{r}\bigg)\bigg) \cdot \sum_{\substack{r_i'|g\\b_i'\in (\mbf{Z}/r_i'\mbf{Z})^{\ast}}}\frac{\prod_{i=1}^{3}\mu(r_i')}{\prod_{i=1}^{3}\phi(r_i')}\sum_{\substack{x_1 + x_2 + x_3= N\\x_i\in \mc{S}_b }}e\bigg(\sum_{i=1}^{3}\frac{b_i'x_i}{r_i'}\bigg).
\end{align*}
We have (using that the Ramanujan sum is multiplicative) that
\begin{align*}
\sum_{\substack{r\le (\log N)^{A}\\(g,r) =1}}\frac{\mu(r)^3}{\phi(r)^3}\sum_{b\in (\mbf{Z}/r\mbf{Z})^{\ast}}e\bigg(\frac{bN}{r}\bigg) &= \sum_{\substack{r\ge 1\\(g,r) =1}}\frac{\mu(r)^3}{\phi(r)^3}\sum_{b\in (\mbf{Z}/r\mbf{Z})^{\ast}}e\bigg(\frac{bN}{r}\bigg)+ O((\log N)^{-A})\\
&=\prod_{\substack{p\\(p,g) =1}}\bigg(1 - \frac{1}{(p-1)^3}\cdot \sum_{b\in (\mbf{Z}/p\mbf{Z})^{\ast}}e\bigg(\frac{bN}{p}\bigg)\bigg)\\
&=\prod_{\substack{p\\(p,g) =1}}\bigg(1 - \frac{p\mbm{1}_{p|N} - 1}{(p-1)^3}\bigg).
\end{align*}
For the second sum, it suffices to prove that 
\[\sum_{\substack{r|g\\b\in (\mbf{Z}/r\mbf{Z})^{\ast}}}\frac{\mu(r)}{\phi(r)}e\bigg(\frac{bx}{r}\bigg) = \mbm{1}_{(x,g) = 1}.\]
Note that it suffices to consider $g = p_1\cdots p_k$ which are distinct primes. The above then factors as 
\[\prod_{i=1}^{k}\bigg(1 -\frac{\sum_{b\in (\mbf{Z}/p_i\mbf{Z})^{\ast}}e(bx/p_i)}{(p_i-1)}\bigg) = \mbm{1}_{(x,g) = 1} \cdot \prod_{i=1}^{k}\frac{p_i}{(p_i-1)}.\]
These together give exactly the stated main term and we have completed the proof. 
\end{proof}

\appendix
\section{Character sum estimates}\label{app:char-sum}

We prove the various character sum estimates that are implicit in \cref{lem:char-sum}. Various input which we require all appear in \cite[Chap.~11,~12]{IK04}.

The first estimate we require is the Weil bound (see \cite[Theorem 11.23]{IK04}).
\begin{theorem}\label{thm:weil}
Let $p$ be prime and $\chi$ be a nontrivial multiplicative character of $\mb{F}_p$ of order $d$. Let $f\in \mb{F}_p[x]$ be a nontrivial polynomial with $m$ distinct roots and which is not a perfect $d$-th power. Then 
\[\bigg|\sum_{x \in \mb{F}_p} \chi(f(x))\bigg| \le (m - 1)\sqrt{p}.\]
\end{theorem}

We next require \cite[Lemma~12.2,~Lemma~12.3]{IK04} which allow one to simplify character sums for cyclic groups of prime power order. The proofs are elementary and rely on a Hensel--type analysis.
\begin{lemma}\label{thm:evenbound}
Let $p$ be prime, $\alpha\ge 1$, and $f,g\in \mbf{Z}[x]$. Then if $q = p^{2\alpha}$, we have 
\[\sum_{y\in \mbf{Z}/q\mbf{Z}}\chi(f(y))e(ag(y)/q) = p^{\alpha}\sum_{\substack{y\in (\mbf{Z}/p^{\alpha}\mbf{Z})\\ h(y)\equiv 0\imod p^{\alpha}}}\chi(f(y))e(ag(y)/q)\]
where:
\begin{itemize}
    \item $h(y) = ag'(y) + bf'/f(y)$
    \item $b \in \mbf{Z}$ is chosen so that $\chi(1 + zp^\alpha) = e(bz/p^\alpha)$
\end{itemize}

Furthermore, if $q = p^{2\alpha + 1}$, then 
\[\sum_{y\in \mbf{Z}/q\mbf{Z}}\chi(f(y))e(ag(y)/q) = p^{\alpha}\sum_{\substack{y\in (\mbf{Z}/p^{\alpha}\mbf{Z})\\ h(y)\equiv 0\imod p^{\alpha}}}\chi(f(y))e(ag(y)/q)G_p(y)\]
where:
\begin{itemize}
    \item $G_p(y) = \sum_{z \in \mbf{F}_p} e_p(d(y)z^2 + h(y)p^{-\alpha}z)$
    \item $h(y) = ag'(y) + bf'/f(y)$
    \item $d(y) = \frac{a}{2}g''(y) + \frac{b}{2}f''/f(y) + (p - 1)\frac{b}{2}(f'/f(y))^2$
    \item $b\in \mbf{Z}$ is chosen so that $\chi(1 + zp^\alpha) = e\big(b\frac{z}{p^{\alpha + 1}} + (p - 1)\frac{bz^2}{2p}\big).$
\end{itemize}
\end{lemma}

We next require the following elementary lemma which bounds the number of solutions to a quadratic modulo an odd prime power. 
\begin{lemma}\label{lem:square-count-odd}
For $a\in \mbf{Z}$, $p$ an odd prime, and $k\ge 1$ an integer, we have that 
\[\bigg|\{x^2\equiv a\imod p^k: x\in \mbf{Z}/p^k\mbf{Z}\}\bigg|\le 2p^{k-\lceil k/2\rceil}.\]
\end{lemma}
\begin{proof}
Let $v_p(a) = 2\ell$; if $2\nmid v_p(a)$, then there are no solutions. Replacing $p^k$ by $p^{k-2\ell}$ and noting that $p^{\ell}|x$, it suffices to prove that if $(a,p) = 1$, then
\[\bigg|\{x^2\equiv a\imod p^k: x\in \mbf{Z}/p^k\mbf{Z}\}\bigg|\le 2.\]
The result is immediate for $k = 1$ and via Hensel's lemma and $(2x,p) = 1$, we find that there are at most $2$ solutions for all $k\ge 1$.
\end{proof}

We are now in position to prove \cref{lem:char-sum} when $p$ is odd; the case when $p = 2$ is more delicate and handled at the end. Note that \cref{lem:char-sum} in the case of odd primes is handled by replacing $t+b_2$ by $t$ and taking $m = b_1-b_2\cdot p^{a_1-a_2}$ in \cref{lem:odd-char}.

\begin{lemma}\label{lem:odd-char}
Let $p$ be an odd prime and $\chi_i$ be primitive characters modulo $p^{a_i}$ with $a_1\ge a_2$. Let $\alpha = \lfloor a_1/2\rfloor$. If $\alpha > 2(a_1-a_2) + v_p(T)$, then 
\[|\mb{E}_{\ell\in \mbf{Z}/p^{a_2}\mbf{Z}}\chi_1(\ell p^{a_1-a_2} + m)\chi_2(\ell)e(\ell T/p^{a_2})|\le \min(1,2p^{-\lceil (\alpha - \gamma)/2\rceil + (a_1-a_2)}).\]
Furthermore, if $a_1 = a_2 = 1$ and $p\nmid T$, we have that 
\[|\mb{E}_{\ell\in \mbf{Z}/p\mbf{Z}}\chi_1(\ell + m)\chi_2(\ell)e(\ell T/p)|\le 3p^{-1/4}.\]
\end{lemma}
\begin{proof}
We first consider the case where $a_1\ge 2$; the initial bound is trivial when $a_1 = 1$. Note that $\chi_2$ naturally lifts to a character on $\mbf{Z}/p^{a_1}\mbf{Z}$. As $p$ is odd, there exists a primitive character $\chi$ modulo $\mbf{Z}/p^{a_1}\mbf{Z}$ such that 
\[\chi_{1} = \chi^{k_1}\text{ and }\chi_2 = \chi^{k_2}.\]
Therefore, we seek to estimate
\[\big|\mb{E}_{\ell\in \mbf{Z}/p^{a_1}\mbf{Z}}\chi((\ell p^{a_1-a_2}+m)^{k_1}\ell^{k_2})e(\ell T/p^{a_2})\big|.\]

Due to \cref{thm:evenbound}, we see that it suffices to bound the number of solutions to $h(x)\equiv 0\imod p^{\alpha}$ where 
\[h(x) = p^{a_1-a_2}T + b\bigg(\frac{k_1p^{a_1-a_2}}{p^{a_1-a_2} x+ m} + \frac{k_2}{x}\bigg).\]
Letting $k_2 = p^{a_1-a_2}k_2'$,
\[h(x) = p^{a_1-a_2}T + b\bigg(\frac{k_1}{p^{a_1-a_2}x + m} + \frac{k_2'}{x}\bigg).\]
Therefore, it suffices to study roots of 
\[h_2(x) = Tx(p^{a_2-a_1}x + m) + b(k_1x + k_2'(p^{a_1-a_2}x + m)):= \ell_1x^2 + \ell_2x + \ell_3\]
in $\mbf{Z}/p^{\alpha - (a_2-a_1)}\mbf{Z}$. Note that $v_p(\ell_1) = v_p(Tp^{a_2-a_1}) = v_p(T) + a_2-a_1$ and by assumption that $v_{p}(\ell_1)<\alpha - (a_2-a_1)$.

If $v_p(\ell_1)\le v_p(\ell_2)$, we may divide by $p^{v_p(\ell_1)}$, and complete the square. In this case, by \cref{lem:square-count-odd} and \cref{thm:evenbound}, we obtain an upper bound on the character sum of 
\[\le 2p^{-\lceil (\alpha - (a_2-a_1)-v_p(\ell_1))/2\rceil}\le 2p^{-\lceil (\alpha - v_p(T))/2\rceil + (a_2-a_1)}.\]

In the remaining case, we may assume that $v_p(\ell_2)<v_{p}(\ell_1)$. We must have that $v_p(\ell_3)\ge v_p(\ell_2)$; otherwise, we trivially see there are no solutions. Thus we are forced to examine 
\[p^{-v_p(\ell_2)}h_2(x) \equiv 0 \imod p^{\alpha-(a_2-a_1) - v_p(\ell_2)}.\]
Since the linear coefficient is not divisible by $p$, we may use Hensel's lemma to obtain that there is at most one solution of the above equation. In this case, we achieve a superior upper bound of 
\[\le p^{-\alpha + (a_2-a_1) + v_p(\ell_2)}\le p^{-\alpha + v_p(T) + 2(a_2-a_1)}\]
on the character sum. We now address the case when $a_1 = a_2 = 1$. As before, there exists a primitive character $\chi$ such that $\chi_1 = \chi^{k_1}$ and $\chi_2 = \chi^{k_2}$ with $k_1, k_2 \in [0, p - 1]$. We have that 
\begin{align*}
\big|\mb{E}_{\ell\in \mb{F}_p}\chi((\ell+m)^{k_1}\ell^{k_2})e(-\ell T/p)\big|^2 &= \mb{E}_{\ell,t\in \mb{F}_p}\chi((\ell+m)^{k_1}(\ell + t + m)^{p-1-k_1}\ell^{k_2}(\ell + t)^{p-1-k_2})e(-t T/p)\\
&\le \mb{E}_{t\in \mb{F}_p}\big|\mb{E}_{\ell\in \mb{F}_p}\chi((\ell+m)^{k_1}(\ell + t + m)^{p-1-k_1}\ell^{k_2}(\ell + t)^{p-1-k_2})\big|.
\end{align*}
If $p\nmid m$, then each term in the above is bounded by $4p^{-1/2}$ unless $t = 0$ or $-m$. When $p|m$, we have to instead consider 
\[\mb{E}_{t\in \mb{F}_p}\big|\mb{E}_{\ell\in \mb{F}_p}\chi((\ell + t)^{2(p-1)-k_1-k_2}\ell^{k_1 + k_2})\big|.\]
If $k_1 + k_2 \neq p-1$, we immediately have the desired result. Finally, in the case where $k_1 + k_2 = p-1$ and $p|m$, we have that $\chi_1 = \bar{\chi_2}$ and thus, we are evaluating 
\[\mb{E}_{\ell\in \mb{F}_p}e(\ell T/p) \cdot \mbm{1}_{p\nmid\ell}\]
which is $-1/p$ if $p\nmid T$.
\end{proof}

We now handle the case where $p = 2$; the bounds for counting solutions over squares is now slightly distinct. 

\begin{lemma}\label{lem:square-count-even}
For $a \in \mbf{Z}$, we have that for all integers $k \ge 1$,
\[\big\{x^2 \equiv a \imod 2^k:x\in \mbf{Z}/2^k\mbf{Z}\big\}\le 4\cdot 2^{k-\lceil k/2\rceil}.\]
\end{lemma}
\begin{proof}
Note that there are no solutions if $2 \nmid v_2(a)$. Thus, we may reduce to $(a,2) = 1$ and prove that 
\[\big\{x^2 \equiv a \imod 2^k:x\in \mbf{Z}/2^k\mbf{Z}\big\}\le 4.\]
Suppose that $2^k|(x^2-y^2)$ and $2\nmid xy$. Then either $2^{k-1}|(x-y)$ or $2^{k-1}|(x+y)$, both of which imply that there are at most $4$ solutions as desired. 
\end{proof}

We finally reach the case of $p = 2$ of \cref{lem:char-sum}. The primary difference with the odd case stems from the fact that the underlying group of characters is no longer cyclic; in addition, we will be rather careless with fixed powers of two in the lemma as they feed harmlessly into the main proof. 
\begin{lemma}\label{lem:char-sum-even}
Let $p=2$ and $\chi_i$ be primitive characters modulo $p^{a_i}$ with $a_1\ge a_2$. Let $\alpha = \lfloor a_1/2\rfloor$ and if $\alpha > 2(a_1-a_2) + v_p(T)$ then 
\[|\mb{E}_{\ell\in \mbf{Z}/2^{a_2}\mbf{Z}}\chi_1(\ell 2^{a_1-a_2} + m)\chi_2(\ell)e(\ell T/2^{a_2})|\le \min(1,32 \cdot 2^{-\lceil (\alpha - v_p(T))/2\rceil + (a_1-a_2)}).\]
\end{lemma}
\begin{proof}
We follow the proof in \cref{lem:odd-char} with slight modifications as required. Note that for the bound not to be vacuous, we require that $5<\lceil \alpha/2\rceil$; this means we may assume that $a_1\ge 21$. As $a_1/2 >2(a_1-a_2)$, we have that $a_2\ge 15$.

As $p = 2$, it is no longer true that $\chi_1$ and $\chi_2$ are powers of a single character. Instead, $(\mbf{Z}/2^n\mbf{Z})^{\ast}\simeq \mbf{Z}/2^{n-2}\mbf{Z}\times \mbf{Z}/2\mbf{Z}$. Using an explicit description of this isomorphism, we may take
\[\chi_1 = \chi^{k_1}\xi^{\eps_1}\text{ and }\chi_2 = \chi^{k_2}\xi^{\eps_2}\]
where $\chi$ is a primitive character modulo $2^{a_1}$ and $\xi(x) = (-1)^{(x - 1)/2}$ for $x$ odd and $\xi(x) = 0$ if $x$ is even (and $\eps_i\in \{0,1\}$).
Noting that $\xi$ is zero precisely when $\chi$ is zero, we may replace $\xi$ with $(-1)^{(x - 1)/2}$. Hence,
\[|\mb{E}_{\ell\in \mbf{Z}/2^{a_2}\mbf{Z}}\chi((\ell 2^{a_1-a_2} + m)^{k_1}\ell^{k_2})e(\ell T'/2^{a_2})|\]
where $T\equiv T' \imod 2^{a_2-2}$. Note that if $v_2(T)\ge a_2-2$, then $a_1/2 >2(a_1-a_2) + (a_2 - 2)$ or $2>3a_1/2 - a_2\ge a_1/2$. This is clearly not possible as otherwise, the initial bound is vacuous, and therefore $v_2(T) = v_2(T')$. The average we consider thus becomes
\[|\mb{E}_{\ell\in \mbf{Z}/2^{a_1}\mbf{Z}}\chi((\ell 2^{a_1-a_2} + m)^{k_1}\ell^{k_2})e(\ell T'/2^{a_2})|.\]

Therefore, we wish to count the roots of 
\[h(x) = 2^{a_1-a_2}T' + b\bigg(\frac{2^{a_1-a_2}\cdot k_1}{2^{a_1-a_2}x + m} + \frac{k_2}{x}\bigg)\]
in $\mbf{Z}/2^{\alpha}\mbf{Z}$. Writing $k_2 = 2^{a_2}k_2'$, we count the roots of 
\[2^{a_2-a_1} h(x) = T' + b\bigg(\frac{k_1}{2^{a_1-a_2}x + m} + \frac{k_2'}{x}\bigg)\]
in $\mbf{Z}/2^{\alpha - a_1 + a_2}\mbf{Z}$. By clearing denominators, it suffices to count the roots of 
\[h_2(x) = T'(2^{a_1-a_2}x + m)x + b(k_1x + k_2'(2^{a_1-a_2}x+m)):= \ell_1 x^2 + \ell_2 x + \ell_3\]
in $\mbf{Z}/2^{\alpha - a_1 + a_2}\mbf{Z}$. Note that the largest power of $2$ dividing the coefficient of $x^2$ is $2^{a_1-a_2+v_p(T')}$. 

First, suppose that $v_2(\ell_1)\le  v_2(\ell_2)$. If $v_2(\ell_3)<v_{2}(\ell_2)$, then there are no roots to $h_2$. Otherwise, we consider roots of
\[2^{-v_2(\ell_1)}(\ell_1 x^2 + \ell_2 x + \ell_3)\]
in $\mbf{Z}/2^{\alpha - a_1 + a_2 -v_{2}(\ell_1)}\mbf{Z}$. If $v_2(\ell_2)>v_2(\ell_1)$ we may complete the square and applying \cref{lem:square-count-even} and \cref{thm:evenbound} , we obtain a bound of 
\[4\cdot 2^{-\lceil (\alpha - a_1 + a_2 -v_{2}(\ell_1))/2\rceil} \le  4\cdot 2^{-\lceil (\alpha - v_2(T))/2\rceil + (a_1 -a_2)}\]
on the character sum. If $v_2(\ell_2) = v_2(\ell_1)$, we may multiply by $4$ and then complete the square. After applying \cref{lem:square-count-even} and \cref{thm:evenbound}, we obtain a bound of 
\[16 \cdot 2^{-\lceil (\alpha - a_1 + a_2 -v_{2}(\ell_1) + 2)/2\rceil} \le 8\cdot 2^{-\lceil (\alpha - v_2(T))/2\rceil + (a_1 -a_2)}\]
on the character sum.

Therefore we are left with the case where $v_2(\ell_1)>v_2(\ell_2)$. Note that we may reduce to the case of $v_2(\ell_3)\ge v_2(\ell_2)$ as otherwise, there are no solution to $h$. Thus, it suffices to consider solutions of 
\[2^{-v_2(\ell_2)}(\ell_1 x^2 + \ell_2 x + \ell_3)\]
in $\mbf{Z}/2^{\alpha - a_1 + a_2 -v_{2}(\ell_2)}\mbf{Z}$. Note that there is a unique solution via Hensel's lemma in this situation and thus by \cref{thm:evenbound}, we obtain a bound of 
\[2^{-\alpha + a_1 - a_2 + v_2(\ell_2)}\le 2^{-\alpha + 2(a_1 - a_2) + v_2(T)}\]
and the character sum. All three bounds we obtained on the character sum are sufficient for this lemma.
\end{proof}
\bibliographystyle{amsplain0}
\bibliography{main.bib}

\end{document}